\documentclass[preprint,3p,times]{elsarticle}

\journal{Computer Methods in Applied Mechanics and Engineering}
\date{\today}

\usepackage{tikz}
\usetikzlibrary{positioning}
\usepackage{amsmath}
\usepackage{amssymb}
\usepackage{graphicx}
\usepackage{relsize}
\usepackage{nicematrix}
\usepackage{algorithm}
\usepackage{algorithmic}
\usepackage{bm}
\usepackage{hyperref}
\usepackage{cleveref}
\usepackage{listings}
\usepackage{nicematrix}
\usepackage{booktabs}
\usepackage{cleveref}
\usepackage{enumitem} 
\Crefname{listing}{code snippet}{code snippet}  
\Crefname{listing}{Code Snippet}{Code Snippet}
\usepackage{todonotes}
\usepackage{algorithm}
\usepackage{algorithmic}
\usepackage{bm}
\usepackage{todonotes}
\usepackage{pythonhighlight}
\usepackage{amsthm}
\usepackage{xcolor}
\definecolor{M1}{RGB}{230, 240, 255}
\definecolor{M2}{RGB}{255, 230, 230}
\definecolor{M3}{RGB}{230, 255, 230}
\definecolor{M4}{RGB}{255, 255, 220}
\usepackage{siunitx}

\newcommand{\ba}{\mathbf a}

\newcommand{\bh}{\mathbf h}
\newcommand{\bj}{\mathbf j}

\newcommand{\bw}{\mathbf w}
\newcommand{\bx}{\mathbf x}
\newcommand{\by}{\mathbf y}
\newcommand{\bz}{\mathbf z}
\newcommand{\bmu}{\boldsymbol \mu}

\newcommand{\calb}{\mathcal B}

\newcommand{\caln}{\mathcal N}

\newcommand{\R}{\mathbb R}

\newtheorem{theorem}{Theorem}[section]

\newtheorem{corollary}[theorem]{Corollary}
\newtheorem{definition}[theorem]{Definition}

\newtheorem{proposition}[theorem]{Proposition}

\newtheorem{remark}[theorem]{Remark}

\newcommand{\model}{ELM-FBPINN with RRQR filtering}
\newcommand{\modelss}{ELM-FBPINNs with RRQR filtering}

\usepackage{xcolor}

\newcommand{\mat}[1]{{\bf #1}}

\begin{document}
\begin{frontmatter}



\title{Local Feature Filtering for Scalable and Well-Conditioned Domain-Decomposed Random Feature Methods}

\author[1]{Jan Willem van Beek}\ead{j.w.v.beek@tue.nl}
\affiliation[1]{organization = {Eindhoven University of Technology},
addressline={PO Box 513}, city={Eindhoven},
postcode={5600 MB},
country={The Netherlands}}
\author[1]{Victorita Dolean\corref{cor1}}\ead{v.dolean.maini@tue.nl}
\cortext[cor1]{Corresponding author}
\author[2]{Ben Moseley} \ead{b.moseley@imperial.ac.uk}
\affiliation[2]{organization = {Imperial College London},
addressline={Exhibition Rd, South Kensington}, city={London},
postcode={SW7 2AZ},
country={United Kingdom}}

\begin{abstract}
Random Feature Methods (RFMs) \cite{chen:2022:BTM} and their variants such as extreme learning machine finite-basis physics-informed neural networks (ELM-FBPINNs) \cite{Anderson:2024:ELM} offer a scalable approach for solving partial differential equations (PDEs) by using localized, overlapping and randomly initialized neural network basis functions to approximate the PDE solution and training them to minimize PDE residuals through solving structured least-squares problems. This combination leverages the approximation power of randomized neural networks, the parallelism of domain decomposition, and the accuracy and efficiency of least-squares solvers. However, the resulting structured least-squares systems are often \emph{severely ill-conditioned}, due to local redundancy among random basis functions and correlation introduced by subdomain overlaps, which significantly affects the convergence of standard solvers.

In this work, we introduce a \emph{block rank-revealing QR (RRQR) filtering and preconditioning} strategy that operates directly on the structured least-squares problem. First, local RRQR factorizations identify and remove redundant basis functions while preserving numerically informative ones, reducing problem size, and improving conditioning. Second, we use these factorizations to construct a right preconditioner for the global problem which preserves block-sparsity and numerical stability. Third, we derive deterministic bounds of the condition number of the preconditioned system, with probabilistic refinements for small overlaps.

We validate our approach on challenging, multi-scale PDE problems in 1D, 2D, and $(2+1)$D, demonstrating reductions in condition numbers by up to eleven orders of magnitude, LSQR convergence speedups by factors of $10$--$1000$, and higher accuracy than both unpreconditioned and additive Schwarz-preconditioned baselines, all at significantly lower memory and computational cost. These results establish RRQR-based preconditioning as a scalable, accurate, and efficient enhancement for RFM-based PDE solvers.
\end{abstract}

\end{frontmatter}

\section{Introduction}
Accurate and efficient numerical solution of partial differential equations (PDEs) is a fundamental challenge in computational science and engineering. Traditional discretization approaches such as finite difference (FDM) and finite element methods (FEM) are well-developed and robust but can face difficulties in high-dimensional, complex-geometry, or multi-scale settings, where accuracy, efficiency, and scalability can be hard to balance.  

In recent years, \emph{scientific machine learning} (SciML) has emerged as a promising alternative. Among SciML methods, \emph{physics-informed neural networks} (PINNs) \cite{Lagaris1998,Raissi2019} recast the PDE solution into a residual minimization problem, using a neural network to directly approximate the solution and enabling a mesh-free solution and easy integration with inverse problems. The idea of neural-network-based PDE solvers dates back to \cite{dissanayake:1994:NTB}, and the approach was extended by \cite{Lagaris1998} to handle boundary and initial value problems by constructing trial functions that satisfy boundary conditions exactly.  

While PINNs offer appealing properties, they exhibit well-known limitations: they suffer from spectral bias \cite{Rahaman2018}, struggle with multi-scale PDEs \cite{Moseley2023,Wang2021d}, and exhibit super-linear growth in training complexity \cite{Moseley2022,Moseley2023}. Domain decomposition strategies have been proposed to address these shortcomings. Conservative PINNs (cPINNs) \cite{jagtap2020conservative} enforce flux continuity across non-overlapping subdomains; extended PINNs (XPINNs) \cite{jagtap2020extended} generalize to arbitrary space-time decompositions; $hp$-variational PINNs ($hp$-VPINNs) \cite{kharazmi2021hp} incorporate piecewise polynomial test functions; and the deep domain decomposition method (DeepDDM) \cite{li2020deep} applies Schwarz iterations. Finite-basis PINNs (FBPINNs) \cite{Moseley2023} combine overlapping Schwarz decomposition with localized neural networks, later enhanced by iterative training techniques \cite{dolean2022finite} and extended to multilevel decomposition \cite{Dolean:MDD:2024}.  

In FBPINNs, the solution domain $\Omega$ is split into overlapping subdomains $\Omega = \cup_{j=1}^J \Omega_j$, each equipped with a neural network $u_j(\mathbf{x},\bm{\theta}_j)$, with $\mathbf{x}$ denoting coordinates in the domain and $\bm{\theta}_j$ denoting the network's learnable parameters. The global solution approximation is constructed as
\(
u(\mathbf{x},\bm{\theta}) = \sum_{j=1}^J \omega_j(\mathbf{x}) \, u_j(\mathbf{x},\bm{\theta}_j),
\)
where $\omega_j(\mathbf{x})$ are smooth, localized window functions forming a partition of unity. This localization improves accuracy and mitigates spectral bias but leaves the method computationally expensive, since all network parameters must be optimized via non-convex gradient descent.

A powerful acceleration idea is to replace subdomain networks with \emph{extreme learning machines} (ELMs) \cite{Huang2006}, where hidden layer weights are randomized and fixed, which reduces training to solving a linear least-squares problem for the output weights. In this way, the PDE solution is approximated with localized, overlapping and randomly initialized neural network basis functions (or \emph{features}) and yields a domain-decomposed variant of the \emph{Random Feature Method} (RFM) \cite{chen:2022:BTM}, which is akin to earlier randomized network paradigms such as random vector functional-links (RVFLs) \cite{pao1992functional} and reservoir computing \cite{Jaeger:2009:RCA}. RFMs have also been combined with structured sampling strategies, such as the SWIM method \cite{datar:2024:PDE}, to improve accuracy and efficiency.

Despite these advantages, RFMs and ELM-FBPINNs face a major challenge: the least-squares system is often \emph{severely ill-conditioned}. Two main factors seem to contribute in a decisive manner: (i) \textbf{Local redundancy}: Within each subdomain, randomly generated features may be nearly linearly dependent; (ii) \textbf{Overlap interference}: Features overlapping between neighboring subdomains introduce correlations that affect conditioning.

Recent work by \cite{Shang_Heinlein_Mishra_Wang_2025} proposed addressing these issues via subdomain singular value decomposition (SVD) for local dimensionality reduction, combined with an additive Schwarz (AS) preconditioner applied to the normal equations of the least-squares system. While this approach can improve solver convergence, it has three significant drawbacks:  
(i) Preconditioning the normal equations squares the condition number $\kappa(\mat{M})$ of the global least-squares matrix, $\mat{M}$,  
(ii) Forming $\mat{M}^\top \mat{M}$ destroys sparsity,  
and (iii) SVD and AS preconditioner assembly incur high memory and computational cost. In this work, we propose a \emph{direct least-squares preconditioning strategy} based on \emph{subdomain rank-revealing QR (RRQR)} factorizations \cite{Chan_1987}. The proposed method has three key advantages:
\begin{enumerate}
    \item \textbf{Feature selection}: Local RRQR identifies and retains numerically informative features while discarding redundant ones, reducing problem size and improving conditioning.
    \item \textbf{Sparse, stable preconditioning}: We assemble a block-structured right preconditioner for the global least-squares problem directly from the filtered QR factors, preserving sparsity and avoiding the stability loss of normal-equation preconditioning.
    \item \textbf{Theoretical insight}: We provide a deterministic analysis of the condition number for the preconditioned system. Probabilistic refinements, based on a random matrix model for the local feature blocks, are given in \ref{appendix:RMT}.
\end{enumerate}

Our contributions are validated through extensive numerical experiments on PDEs in 1D, 2D, and $(2+1)$D. We demonstrate up to eleven orders of magnitude improvements in conditioning, solver convergence speedups by factors of $10$--$1000$, and higher accuracy over both unpreconditioned and AS-preconditioned baselines, at significantly lower memory and computational cost.  

The remainder of the paper is organized as follows. In \Cref{sec:background} we review RFMs, ELM-FBPINNs, and relevant preconditioning strategies, and recall the RRQR factorization. In \Cref{sec:contribution} we present our proposed RRQR filtering and preconditioning method, together with a theoretical analysis of conditioning in the presence of overlap. In \Cref{sec:numerical-results} we evaluate the method on a range of PDE problems, including strong and weak scaling studies, and compare against established baselines. We conclude this section with a discussion of limitations and directions for future work.

\section{Background} 
\label{sec:background}
In this section we begin by formulating the Random Feature Method (RFM) for solving boundary value problems (BVPs). The goal is to approximate the solution of a linear PDE using a neural ansatz that combines randomized neural network basis functions (or features) with localized support. By fixing the random feature parameters and optimizing only the output weights, we reduce the problem to a structured least-squares problem. This section details the mathematical setup, introduces the learning task, and reformulates this task into a least-squares problem in matrix form to set the stage for subsequent analysis of conditioning and preconditioning strategies.

\subsection{RFM for BVP a.k.a. the ELM-FBPINN method}\label{sec:ELM-FBPINN-background}
We start with a PDE defined in a domain $\Omega \subset \R^d $
\[ \mathcal{N}[u(\bx)] = f(\bx), \ \ \bx \in \Omega, \]
under boundary conditions given by 
\[ \mathcal{B}_l[u(\bx)] = g_l(\bx), \ \ \bx \in \Gamma_l \subset \partial \Omega, \ \ l=1,\dots,L, \]
where $\caln$ and the $\calb_l$ are assumed to be linear operators and $f$ and $g_l$ are forcing and boundary functions. We approximate the solution $u$ with a piecewise neural network, $\hat{u} \approx u,$ 
\begin{equation}
\label{eq:elm_solution}
\hat{u}(\bx,\ba) = \sum_{j = 1}^J \ \omega_j(\bx) \sum_{k = 1}^K \ a_{jk} \phi_{jk}(\bx,\theta_{jk}),
\end{equation}
with $\phi_{jk}(\bx,\theta_{jk})$ being any parameterised neural network function with a scalar output. In this paper, without loss of generality, we use shallow fully connected networks of the form
\begin{equation}
\label{eq:fcn_depth_h}
\begin{aligned}
&\bz_{j}^{(0)} = \bx, \\
&\bz_{j}^{(l)} = \Xi\left( W_{j}^{(l)} \bz_{j}^{(l-1)} + b_{j}^{(l)} \right), \quad l = 1,\dots,h-1, \\
&\phi_{jk} = \Xi\left( \bw_{jk}^{\top} \bz_{j}^{(h-1)} \right),
\end{aligned}
\end{equation}
where $\Xi$ is a nonlinear activation function applied elementwise, and the network parameters are given by $\theta_{jk} = \{W_{j}^{(l)}, b_{j}^{(l)}\}_{l=1}^{h-1} \cup \{\bw_{jk}\}$. The $\omega_j$ are \emph{window functions}, with local support on subdomains $\Omega_j$ that cover $\Omega$. These subdomains are allowed to overlap. The $a_{jk}$ are the output weights of each localized neural network.

The parameters $\theta_{jk}$ are now randomized and fixed, so only the coefficients $a_{jk}$ remain as free parameters. The idea of randomizing a subset of the parameter space gained prominence in the machine learning space with the publication of \cite{Rahimi_Brecht_2007}. 
More recently, the combination of feature randomization and locally supported window functions, the latter inspired by domain-decomposition methods, has been explored in \cite{Anderson:2024:ELM}, \cite{chen:2022:BTM}, \cite{Shang_Heinlein_Mishra_Wang_2025}.

The model parameters $\bm{a} = \{a_{jk}\}$ are determined by minimizing a loss functional that penalizes deviations from both the PDE and its boundary conditions at a set of collocation and boundary points, in a similar fashion to PINNs \cite{Lagaris1998,Raissi2019}. Specifically, we define the loss function as
\begin{equation}\label{eq:ELMFBPINN_loss}
    \mathcal{L}(\bm{a}) = \frac{1}{N_I} \sum_{n=1}^{N_I} \left( \mathcal{N}[\hat{u}(\bx_n, \bm{a})] - f(\bx_n) \right)^2 
    + \sum_{l=1}^{L} \frac{\lambda_l}{N_B^{(l)}} \sum_{i=1}^{N_B^{(l)}} \left( \mathcal{B}_l[\hat{u}(\bx_i^{(l)}, \bm{a})] - g_l(\bx_i^{l)}) \right)^2.
\end{equation}
Here, $\{\bx_n\}_{n=1}^{N_I}$ are collocation points sampled over the interior of the domain $\Omega$, and $\{\bx_i^{(l)}\}_{i=1}^{N_B^{(l)}}$ are collocation points sampled on the boundary subsets $\Gamma_l \subset \partial\Omega$, associated with boundary operators $\mathcal{B}_l$. The weights $\lambda_l$ are hyperparameters used to balance the strength of each constraint in the optimization problem.

Thanks to the linearity of both the differential operator $\mathcal{N}$ and the boundary operators $\mathcal{B}_l$, the loss functional \eqref{eq:ELMFBPINN_loss} is quadratic in the unknown coefficients $\bm{a}$. This allows us to reformulate the problem as a structured linear least-squares problem:
\begin{equation} \label{eq:ELMFBPINN_ls_all_constraints}
    \min_{\bm{a} \in \mathbb{R}^{JK}} \left\| \mat{N} \bm{a} - \mathbf{f} \right\|^2 + \left\| \mat{B} \bm{a} - \mathbf{g} \right\|^2,
\end{equation}
where $\mat{N}$ and $\mat{B}$ are the system matrices corresponding to the PDE and boundary residuals, and $\mathbf{f}$ and $\mathbf{g}$ are the vectors of evaluated right-hand side functions at the interior and boundary collocation points, respectively.

\Cref{eq:ELMFBPINN_ls_all_constraints} can be expressed as a single global least-squares problem given by

\begin{equation} \label{eq:ELMFBPINN_ls}
    \min_{\bm{a} \in \mathbb{R}^{JK}} \left\| \mat{M} \bm{a} - \bh \right\|^2,
\end{equation}
where $\mat{M}$ and $\bh$ are defined as
\[
\mat{M} = \begin{bmatrix} \mat{N} \\ \mat{B} \end{bmatrix}, \quad 
\bh = \begin{bmatrix} \mathbf{f} \\ \mathbf{g} \end{bmatrix}.
\]
This global problem combines the residual minimization over the PDE and the boundary conditions into a single linear system, where each row block of $\mat{M}$ and $\bh$ corresponds to either interior or boundary collocation points.

\paragraph{Indexing conventions and notations} Before discussing the structure of the system matrices, we introduce some convenient indexing conventions:

\begin{itemize}
    \item Let $\ell = (j - 1)K + k$ map the pair $(j, k)$ from \Cref{eq:elm_solution} to a single global neuron index.
    \item Let $p = \sum_{k = 1}^{l - 1} N_B^{(k)} + i$ map the pair $(i, l)$ from \Cref{eq:ELMFBPINN_loss} to a global index over boundary collocation points.
    \item Let $N_B = \sum_{l = 1}^L N_B^{(l)}$ denote the total number of boundary points.
\end{itemize}

We now define the vectors and matrices involved in the least-squares problem \eqref{eq:ELMFBPINN_ls}:

\begin{itemize}
    \item The vector $\mathbf{f} = \{\frac{1}{\sqrt{N_I}}f(\bx_n)\}_{n = 1}^{N_I}$ contains evaluations of the PDE right-hand side at interior collocation points.
    \item The vector $\mathbf{g} = \{g_p\}_{p = 1}^{N_B}$, where $g_p = \sqrt{\frac{\lambda_l}{N_B^{(l)}}} g_l(\bx_i^{(l)})$, contains evaluations of the boundary data at boundary collocation points.
    \item The parameter vector is $\bm{a} = \{a_\ell\}_{\ell = 1}^{JK}$, where $a_\ell = a_{j,k}$.
\end{itemize}

Matrices $\mat{N} \in \mathbb{R}^{N_I \times JK}$ and $\mat{B} \in \mathbb{R}^{N_B \times JK}$ are defined component-wise by
\[
\begin{array}{l}
N_{n,\ell} = \frac{1}{\sqrt{N_I}}\mathcal{N}\left[(\omega_j \phi_{jk})(\bx_n)\right], \quad B_{p,\ell} = \sqrt{\frac{\lambda_l}{N_B^{(l)}}}\mathcal{B}_l\left[(\omega_j \phi_{jk})(\bx_i^{(l)})\right],
\end{array}
\]
where $\ell = (j - 1)K + k$, $n = 1, \ldots, N_I$, and $p = 1, \ldots, N_B$.

\paragraph{Domain decomposition structure}
Since the window functions $\omega_j$ have local support, the matrix $\mat{M} \in \mathbb{R}^{(N_I+N_B) \times JK}$ is sparse: the entry $M_{r,\ell}$ is nonzero only when $\omega_j(\bx_r) \neq 0$. 

A common problem in RFM is heavy linear dependence between the columns of $\mat{M}$ and consequently a large condition number. We want to exploit the sparsity of $\mat{M}$ to improve on this in an efficient way. To this end, we define the index sets
\[
I_j = \{r \mid \omega_j(\bx_r) \neq 0\}, \quad K_j = \{(j-1)K + 1, \ldots, jK\},
\]
where $I_j$ corresponds to interior and boundary points influenced by subdomain $j$, and $K_j$ to global neuron indices associated with that subdomain.

Let $\mat{M}_j = \mat{M}|_{I_j \times K_j}$ denote the submatrix of $\mat{M}$ restricted to $I_j \times K_j$. Introducing restriction matrices $\mat{V}_j \in \mathbb{R}^{|I_j| \times (N_I+N_B)}$ and $\mat{W}_j \in \mathbb{R}^{K \times JK}$, we express $\mat{M}_j$ as
\[
\mat{M}_j = \mat{V}_j \mat{M} \mat{W}_j^T.
\]
This yields the decomposition
\begin{equation} \label{DomainDecomp}
\mat{M} = \sum_{j = 1}^J \mat{V}_j^T \mat{M}_j \mat{W}_j.
\end{equation}


This decomposition \eqref{DomainDecomp} highlights the localized, block-sparse structure of the system matrix $\mathbf{M}$ induced by the use of compactly supported window functions. An example plot of this block-sparse structure is shown in \Cref{fig:harmonic-oscillator-condition}. While this sparsity is beneficial for computational efficiency, it does not fully mitigate the adverse effects of feature redundancy and near-linear dependence among columns, which can lead to ill-conditioning of the system.  To address this, several preconditioning strategies have been proposed in the literature that aim to exploit the domain-decomposed structure of $\mathbf{M}$. 

\subsection{Preconditioning the RFM via domain decomposition}
In this section, we briefly review the preconditioner introduced in \cite{Shang_Heinlein_Mishra_Wang_2025}, which is based on the additive Schwarz method applied to the normal equations. Starting from the decomposition of the system matrix into localized contributions,
\[
\mat{M} = \sum_{j=1}^J \mat{V}_j^T \mat{M}_j \mat{W}_j,
\]
where $\mat{M}_j$ is the submatrix associated with subdomain $\Omega_j$ and $\mat{V}_j$, $\mat{W}_j$ are restriction operators as introduced in \eqref{DomainDecomp}, the method proceeds by forming the Gramian matrix $\mathbf{A} = \mat{M}^T \mat{M}$. From this, one defines local Gramian blocks
\[
\mat{A}_j = \mat{W}_j \mat{A} \mat{W}_j^T = \mat{M}_j^T \mat{M}_j,
\]
which correspond to diagonal sub-blocks of $\mathbf{M}^T \mathbf{M}$ associated with each subdomain. After the application of a PCA to eliminate small singular values and reducing degrees of freedom, leading to a first improvement system conditioning, these blocks are used and assembled into the additive Schwarz preconditioner
\[
\mat{A}^{-1}_{\text{AS}} = \sum_{j=1}^J \mat{W}_j \mat{A}_j^{-1} \mat{W}_j^T.
\]

Although effective in improving convergence in certain settings, a key drawback of this approach is that it operates on the normal equations rather than the original least-squares problem. Forming the Gramian matrix $\mat{M}^T \mat{M}$ explicitly can lead to numerical difficulties: its condition number is the square of that of $\mat{M}$, i.e., $\kappa(\mat{M}^T \mat{M}) = \kappa(\mat{M})^2$, potentially exacerbating round-off errors and amplifying the effect of poorly scaled features.

To overcome the limitations associated with preconditioning the normal equations, we seek alternative strategies that operate directly on the original least-squares system. Beyond numerical stability, this shift is also driven by computational considerations: storing the Gramian matrix $\mathbf{M}^T \mathbf{M}$ can be costly, both in time and memory, and it typically destroys the sparsity pattern of $\mathbf{M}$. In contrast, the matrix $\mathbf{M}$ itself, due to the local support of the window functions, is often highly sparse, and preserving this sparsity is crucial for scalability in high-dimensional settings.

\subsection{Rank revealing $QR$ factorization}\label{sec:rrqr}
To exploit this structure efficiently, we turn to the rank-revealing QR (RRQR) factorization as an algebraic tool for dimension reduction and feature selection within each subdomain. Compared to singular value decomposition (SVD) also used in \cite{Shang_Heinlein_Mishra_Wang_2025}, RRQR offers a favorable trade-off: while both methods have $\mathcal{O}(mn^2)$ complexity for a $m \times n$ matrix with $m \geq n$, RRQR is typically more numerically stable, faster in practice, amenable to sparse input, and returns actual columns of the matrix rather than linear combinations, enhancing interpretability and enabling localized filtering. In the next section, we introduce the formal definition of RRQR and outline how it forms the basis of our preconditioning approach.

The RRQR-based strategy introduced here filters out redundant features within each subdomain, 
leading to a preconditioner that acts directly on the original least-squares system while 
preserving sparsity. In our setting, the RRQR is applied locally to each matrix $\mat{M}_j$ 
rather than to the global matrix $\mat{M}$.

Given a matrix $\mat{P} \in \mathbb{R}^{m \times n}$ with $m \ge n$, a 
$\sigma$-\emph{rank-revealing QR} (RRQR) factorization permutes the columns of $\mat{P}$ so that
\[
\mat{P} \, \Pi = \mat{Q} 
\begin{pNiceArray}{cc}
    \mat{R}_{11} & \mat{R}_{12} \\
     & \mat{R}_{22}
\end{pNiceArray},
\]
where $\mat{Q}$ has orthonormal columns, $\mat{R}_{11}$ is square and well-conditioned, 
and the trailing block $\mat{R}_{22}$ has norm at most $\sigma$. 
Algorithms such as the strong RRQR of Gu and Eisenstat~\cite{Gu_Eisenstat_1996} produce 
such a factorization, effectively identifying a well-conditioned basis for the dominant 
column space of $\mat{P}$ while discarding directions whose contribution is uniformly small.

This makes RRQR a natural tool for our local filtering step: it retains only the most 
informative features in each subdomain, reducing redundancy without compromising stability. 
Interestingly, in certain randomized settings -- such as when $\mat{P}$ arises from features of 
randomly weighted neural networks -- these well-conditioned bases also tend to emerge with high 
probability. Such probabilistic refinements are presented in \ref{appendix:RMT}.

\section{QR-based filtered RFM preconditioner} 
\label{sec:contribution}

This section presents our main theoretical and algorithmic contribution: a 
subdomain-wise filtering strategy based on $\sigma$-rank-revealing QR (RRQR) 
factorizations, designed to address the ill-conditioning that arises in 
least-squares problems from Random Feature Methods (RFM) with overlapping 
subdomains. The method has two goals:
\begin{itemize}
\item \textbf{Remove local redundancy:} Within each subdomain, excess random features 
      can become nearly linearly dependent, inflating the condition number of the 
      local system.
\item \textbf{Control overlap interference:} Even well-conditioned subdomains can 
      interact unfavorably through overlap, further degrading the conditioning of 
      the global system.
\end{itemize}

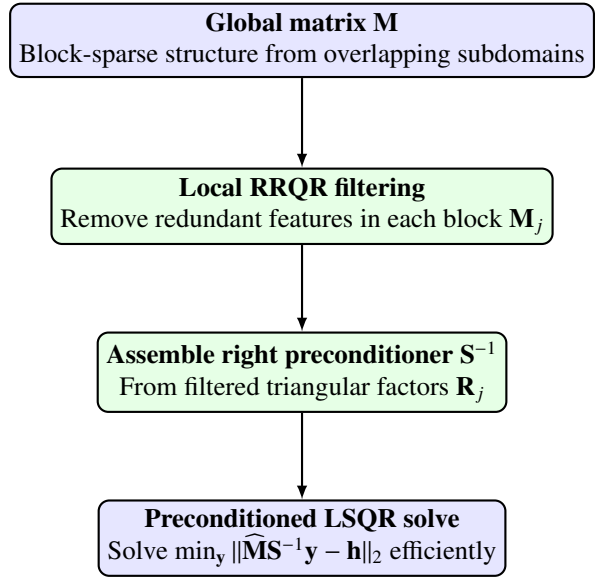
\begin{figure}[htbp]
\centering
\begin{tikzpicture}[node distance=2.2cm,>=latex,thick]

\tikzstyle{block} = [rectangle, draw, fill=blue!10, text centered, rounded corners, minimum height=2em, minimum width=5cm]
\tikzstyle{smallblock} = [rectangle, draw, fill=green!10, text centered, rounded corners, minimum height=2em, minimum width=4cm]
\tikzstyle{arrow} = [->, thick]

\node[block,align=center] (M) {\textbf{Global matrix $\mat{M}$}\\ 
Block-sparse structure from overlapping subdomains};
\node[smallblock, below of=M,align=center] (rrqr) {\textbf{Local RRQR filtering}\\
Remove redundant features in each block $\mat{M}_j$};
\node[smallblock, below of=rrqr,align=center] (prec) {\textbf{Assemble right preconditioner $\mat{S}^{-1}$}\\
From filtered triangular factors $\mat{R}_j$};
\node[block, below of=prec,align=center] (solve) {\textbf{Preconditioned LSQR solve}\\
Solve $\min_\by \| \widehat{\mat{M}} \mat{S}^{-1} \by - \mathbf{h} \|_2$ efficiently};

\draw[arrow] (M) -- (rrqr);
\draw[arrow] (rrqr) -- (prec);
\draw[arrow] (prec) -- (solve);

\end{tikzpicture}
\caption{Workflow of the proposed RRQR filtering and preconditioning method.}
\label{fig:workflow_teaser}
\end{figure}

Our approach applies RRQR locally to each subdomain matrix $\mat{M}_j$ to retain 
only the most informative columns---those corresponding to singular values above 
a threshold $\sigma$---thereby reducing redundancy while preserving sparsity. The 
triangular factors from these local decompositions are then assembled into a 
sparse right preconditioner $S^{-1}$ for the global least-squares problem.

Figure~\ref{fig:workflow_teaser} summarizes the overall process: 
starting from the block-sparse global matrix $\mat{M}$, we perform local RRQR 
filtering, assemble the preconditioner, and solve the preconditioned system 
efficiently with an iterative least-squares solver. A deterministic bound on the 
condition number of the preconditioned system, which captures the effect of 
overlap size on subdomain coupling, is given later in this section. 
Probabilistic refinements of this bound, using a random matrix model for the local 
QR factors, are presented in \ref{appendix:RMT}.

We now describe the local filtering, assembly of the global preconditioner, least squares solver, and condition number estimates in detail.

\subsection{Local filtering} The RRQR decomposition permutes the columns of a matrix to expose a well-conditioned 
leading block while ensuring that the trailing block is small in spectral norm. 
Applying this locally to each $\mat{M}_j$ allows us to identify a numerically full-rank 
basis for its column space, discard directions below the $\sigma$ threshold, and keep 
the local sparsity pattern intact. The process is summarized in 
Algorithm~\ref{alg:RRQR_filter}.
\begin{center}
\setlength{\textfloatsep}{0pt}
\setlength{\floatsep}{0pt}
\setlength{\intextsep}{0pt}
\begin{algorithm}
    \caption{$\sigma$-RRQR filtering method}\label{alg:RRQR_filter}
    \begin{algorithmic}
        \STATE Inputs: block-sparse matrix $\mat{M}$, subdomain count $J$, threshold $\sigma > 0$
        \FOR{$j = 1, \dots, J$}
          \STATE Compute $\sigma$-RRQR factorization of $\mat{M}_j$:
            \[
             \mat{M}_j \, \Pi_j = \mat{Q}_j
             \begin{pNiceArray}{cc}
              \mat{R}_j & * \\
                        & \mat{T}_j
             \end{pNiceArray}
            \]
          \STATE Retain columns associated with $\mat{R}_j$ to form $\widehat{\mat{M}}_j$
        \ENDFOR
        \STATE Assemble $\widehat{\mat{M}}$ from all $\widehat{\mat{M}}_j$
        \STATE Output: global filtered matrix $\widehat{\mat{M}}$
    \end{algorithmic}
\end{algorithm}
\end{center}
Note in our algorithm, the threshold $\sigma$ for each $\mat{M}_j$ is defined relative to the leading diagonal value of $\mat{R}_j$.


\subsection{Assembly of the preconditioner} Recall from the domain decomposition identity~\eqref{DomainDecomp} that the global matrix $\mat{M}$ can be written as
\[
\mat{M} = \sum_{j = 1}^J \mat{V}_j^T \mat{M}_j \mat{W}_j.
\]
After applying local QR factorization and filtering to each subdomain, i.e., $\widehat{\mat{M}}_j = \widehat{\mat{Q}}_j \widehat{\mat{R}}_j$, we obtain the global filtered decomposition
\[
\widehat{\mat{M}} = \sum_{j = 1}^J \mat{V}_j^T \widehat{\mat{Q}}_j \widehat{\mat{R}}_j \widehat{\mat{W}}_j.
\]
This suggests defining a preconditioner that operates directly on the original system and not on the normal equation without requiring full matrix assembly, hence being suitable for parallel computations. Specifically, we define the (right) preconditioner
\[
\mat{S}^{-1} = \sum_{j = 1}^J \widehat{\mat{W}}_j^T \widehat{\mat{R}}_j^{-1} \widehat{\mat{W}}_j.
\]
Because each matrix $\widehat{\mat{R}}_j$ is upper triangular, applying $\mat{S}^{-1}$ involves only inexpensive back-substitution. Moreover, the preconditioner never needs to be explicitly formed, only local factors $\widehat{\mat{R}}_j$ and selection matrices $\widehat{\mat{W}}_j$ are needed.

\begin{proposition}
\label{prop:prec}
The right preconditioned matrix simplifies to
\begin{equation}\label{eq:preconditioner}
\mat{Q}:= \widehat{\mat{M}} \mat{S}^{-1} = \sum_{j = 1}^J \mat{V}_j^T \widehat{\mat{Q}}_j \widehat{\mat{W}}_j.
\end{equation}
\end{proposition}

\begin{proof}
Substituting the factorizations $\widehat{\mat{M}}_j = \widehat{\mat{Q}}_j \widehat{\mat{R}}_j$ into the definition of $\widehat{\mat{M}}$ gives
\[
\widehat{\mat{M}} \mat{S}^{-1} = \sum_{j=1}^J \mat{V}_j^T \widehat{\mat{Q}}_j \widehat{\mat{R}}_j \widehat{\mat{W}}_j \widehat{\mat{W}}_j^T \widehat{\mat{R}}_j^{-1} \widehat{\mat{W}}_j = \sum_{j=1}^J \mat{V}_j^T \widehat{\mat{Q}}_j \widehat{\mat{W}}_j,
\]
using the identity $\widehat{\mat{R}}_j \widehat{\mat{R}}_j^{-1} = \mat{I}$ on the range of $\widehat{\mat{W}}_j$.
\end{proof}

This formulation isolates the orthogonal components $\widehat{\mat{Q}}_j$, which are numerically stable, and removes the influence of the potentially small singular values in $\widehat{\mat{R}}_j$, which would otherwise degrade the conditioning of the system. As a result, the preconditioned system is better suited for iterative solvers. Furthermore, the preconditioned system $\widehat{\mat{M}} \mat{S}^{-1}$ has the same block structure and sparsity pattern as $\mat{M}$ filtered on its dominant columns (as shown in \Cref{fig:harmonic-oscillator-condition}). Given this block structure is sparse, this allows us to use efficient sparse linear solvers to solve the preconditioned system.

\subsection{Least-squares solver} Having explained our filtering strategy and global preconditioner, our full workflow (as summarized by \Cref{fig:workflow_teaser}) for solving the original least-squares system is shown in \Cref{alg:RRQR_solver}.

\begin{center}
\setlength{\textfloatsep}{0pt}
\setlength{\floatsep}{0pt}
\setlength{\intextsep}{0pt}
\begin{algorithm}
    \caption{ELM-FBPINN least-squares solver with RRQR filtering and preconditioning}\label{alg:RRQR_solver}
    \begin{algorithmic}
        \STATE Inputs: sparse matrix $\mat{M}$ and vector $\bh$ in global least-squares problem $\left\| \mat{M} \ba - \bh \right\|^2$ (\Cref{eq:ELMFBPINN_ls}), number of subdomains $J$, relative threshold $\sigma > 0$
        \STATE Compute $\sigma$-RRQR factorization, $\widehat{\mat{M}}$, of $\mat{M}$, filtering dominant columns, using \Cref{alg:RRQR_filter}
        \STATE Compute right-side preconditioned system $\widehat{\mat{M}} \mat{S}^{-1}$ using \Cref{eq:preconditioner}
        \STATE Solve preconditioned least-squares problem $\min_\by \left\| \widehat{\mat{M}} \mat{S}^{-1} \by - \bh \right\|^2$ using a sparse linear solver (e.g. LSQR \cite{Paige1982})
        \STATE Compute least-squares solution for dominant columns $\hat{\ba} = \mat{S}^{-1} \by$, using sparse back-substitution
        \STATE Output: global least-squares solution, $\ba$, with values
        \[
        \ba = 
        \begin{cases}
        \hat{\ba} & \text{at dominant columns} \\
        0 & \text{otherwise}.
        \end{cases}
        \]
    \end{algorithmic}
\end{algorithm}
\end{center}
\subsection{Condition number estimate for overlapping subdomains}

In this section we will analyze the spectral properties of the preconditioned matrix from \Cref{prop:prec} in the simplified case of \( J \) overlapping subdomains seen as consecutive blocks in one dimension or that may arise from a strip-wise decomposition in higher dimensions.

Let \( J \) be the number of overlapping subdomains. Each subdomain contributes a block of \( K \) features. The matrix \( \mat{Q} \in \mathbb{R}^{N \times JK} \) is composed of column blocks $\widehat{\mat{Q}}_j$  corresponding to subdomains which are in turn subdivided into subblocks as follows:
\begin{itemize}
  \item Each subdomain \( j \) contributes a local feature matrix \( \mat{Q}_{jj} \in \mathbb{R}^{k_j \times K} \) for its interior,
  \item Overlapping regions between adjacent domains \( j \) and \( j+1 \) are represented by matrices \( \mat{Q}_{j,j+1}, \mat{Q}_{j+1,j} \in \mathbb{R}^{\ell_j \times K} \),
  \item The columns of each $\widehat{\mat{Q}}_j$ are orthonormal.
\end{itemize}

Then \( \mat{Q} \) has the block structure:
\begin{equation}
\label{eq:Q}
\mat{Q} = 
\begin{bmatrix}
\mat{Q}_{11} & 0 & 0 & \cdots & 0 \\
\mat{Q}_{21} & \mat{Q}_{22} & 0 & \cdots & 0 \\
0 & \mat{Q}_{32} & \mat{Q}_{33} & \cdots & 0 \\
\vdots & \ddots & \ddots & \ddots & \vdots \\
0 & \cdots & 0 & 0 & \mat{Q}_{JJ}
\end{bmatrix}
\in \mathbb{R}^{N \times JK}
\end{equation}
Given that the columns of each $\widehat{\mat{Q}}_j$ are orthonormal and that columns of different subdomains are disjoint, we compute:
\[
\mat{Q}^\top \mat{Q} = \mat{I}_{JK} + \mat{E}
\]
where \( \mat{E} \) contains off-diagonal interactions due to overlaps. Specifically, \( \mat{E} \) is a block tridiagonal matrix with:
\[
\mat{E}_{j,j+1} = \mat{Q}_{j+1,j}^\top \mat{Q}_{j,j}, \quad \mat{E}_{j+1,j} = \mat{E}_{j,j+1}^\top
\]
and all diagonal blocks are the identity:
\[
(\mat{Q}^\top \mat{Q})_{jj} = \mat{I}_K \quad \text{for all } j.
\]

So the matrix \( \mat{Q}^\top \mat{Q} \) has the form:
\[
\mat{Q}^\top \mat{Q} = 
\begin{bmatrix}
\mat{I} & \mat{A}_1 & 0 & \cdots & 0 \\
\mat{A}_1^\top & \mat{I} & \mat{A}_2 & \ddots & \vdots \\
0 & \mat{A}_2^\top & \mat{I} & \ddots & 0 \\
\vdots & \ddots & \ddots & \ddots & \mat{A}_{J-1} \\
0 & \cdots & 0 & \mat{A}_{J-1}^\top & \mat{I}
\end{bmatrix}
\]
with \( \mat{A}_j := \mat{Q}_{j+1,j}^\top \mat{Q}_{j,j} \in \mathbb{R}^{K \times K} \).

\begin{theorem}[Condition Number Bound for the Overlapping Block Matrix]
\label{th:cond-bound}
Let \( \mat{Q} \in \mathbb{R}^{N \times JK} \) be a block-structured matrix constructed from \( J \) overlapping subdomains with orthonormal columns defined as in \Cref{eq:Q}. 
Assume further that:
\[
\| \mat{A}_j \|_2 \le \alpha \quad \text{for all } j.
\]
Then the condition number of \( \mat{Q} \) satisfies:
\[
\kappa(\mat{Q}) \le \left( \frac{1 + 2\alpha}{1 - 2\alpha} \right)^{1/2}, \quad \text{provided } 2\alpha < 1.
\]
\end{theorem}

\begin{proof}

By the symmetry of \( \mat{E} \), its eigenvalues are real. Using the definition, the submultiplicativity of the spectral norm and Cauchy-Schwarz inequality, we can easily estimate norm of \( \mat{E} \):
\[
\| \mat{E} \|_2 \le  2\alpha.
\]
Thus, the eigenvalues \( \lambda \) of \( \mat{Q}^\top \mat{Q} = \mat{I} + \mat{E} \) satisfy
\[
1 - 2\alpha \le \lambda \le 1 + 2\alpha.
\]
Therefore, the condition number of \( \mat{Q} \) in the spectral norm is
\[
\kappa(\mat{Q}) = \sqrt{ \kappa( \mat{Q}^\top \mat{Q} ) } \le \left( \frac{1 + 2\alpha}{1 - 2\alpha} \right)^{1/2},
\]
as long as \( 1 - 2\alpha > 0 \), i.e., \( 2\alpha < 1 \).
\end{proof}

\begin{remark}
The matrix \( \mat{Q}^\top \mat{Q} \) is always symmetric and positive semidefinite, since it is a Gram matrix. All its eigenvalues are non-negative.

However, the bound above is only meaningful when \( \mat{Q}^\top \mat{Q} \) is \emph{strictly} positive definite and a stricter inequality holds.
Since we estimate \( \lambda_{\max}(\mat{E}) \le 2\alpha \), the condition \( 2\alpha < 1 \) ensures invertibility of \( \mat{Q}^\top \mat{Q} \) and validity of the condition number estimate.

The value \( \alpha = \max_j \| \mat{Q}_{j+1,j}^\top \mat{Q}_{j,j} \|_2 \) measures the strength of coupling between adjacent subdomains. The smaller the overlap-induced interactions, the better conditioned the global system is — which is somewhat counterintuitive from the domain decomposition perspective.
\end{remark}

\begin{remark}
In the special case of only two overlapping blocks (\( J = 2 \)), the matrix \( \mat{Q}^\top \mat{Q} \) has the form:
\[
\mat{Q}^\top \mat{Q} =
\begin{bmatrix}
\mat{I} & \mat{A}_1 \\
\mat{A}_1^\top & \mat{I}
\end{bmatrix},
\]
The eigenvalues of \( \mat{Q}^\top \mat{Q} \) are then \( 1 \pm \sigma_i(\mat{A}_1) \), and the condition number becomes:
\[
\kappa(\mat{Q}) = \left( \frac{1 + \| \mat{A}_1 \|_2}{1 - \| \mat{A}_1 \|_2} \right)^{1/2}, \quad \text{valid whenever } \| \mat{A}_1 \|_2 < 1.
\]
Hence, the condition for positive definiteness is simply \( \| \mat{A}_1 \|_2 < 1 \), which is strictly weaker than \( 2\alpha < 1 \). This makes the two-block case less restrictive and ensures well-conditioning as long as we have positive definiteness \( \mat{Q}^\top \mat{Q} \succ 0 \).
\end{remark}

Theorem~\ref{th:cond-bound} shows that, in the worst case, the conditioning of the 
preconditioned system depends only on the strength of nearest-neighbour interactions 
between subdomains, measured by \(\alpha\). If these interactions are weak (small $\alpha$), the global system remains well 
conditioned regardless of the number of subdomains or total features.  
Conversely, large overlap couplings drive $\alpha$ towards the $2\alpha < 1$ threshold, 
causing rapid growth in $\kappa(\mat{Q})$.  

The RRQR filtering directly influences $\alpha$.  
Recall that $\mat{A}_j = \mat{Q}_{j,j+1}^\top \mat{Q}_{j,j}$ involves basis vectors 
from two adjacent subdomains, restricted to the overlap region.  
When a $\sigma$-RRQR is applied to each local matrix $\mat{M}_j$, all columns 
associated with singular values smaller than $\sigma$ are discarded.  
This eliminates directions that are both weakly represented locally and highly 
correlated with features from the neighbouring subdomain.  
Consequently, the remaining columns $\widehat{\mat{M}}_j$ --- and their orthonormal 
factors $\widehat{\mat{Q}_{j}}$ --- exhibit reduced cross-interface correlation, 
leading to smaller $\|\mat{A}_j\|_2$.  
In qualitative terms, a larger $\sigma$ produces a more aggressively filtered 
basis, which reduces $\alpha$ and thus tightens the bound in 
Theorem~\ref{th:cond-bound}, at the cost of discarding more features.  
Quantitative estimates of $\|\mat{A}_j\|_2$ are obtained in ~\ref{appendix:RMT} using probabilistic models; a deeper theoretical analysis of these couplings lies beyond the scope of the present paper.

\section{Numerical experiments}
\label{sec:numerical-results}

In this section we assess the numerical performance of our proposed subdomain RRQR filtering and preconditioning strategy. We study how it affects the convergence, efficiency, and accuracy of ELM-FBPINNs when solving boundary and initial value problems (BVPs and IVPs) in one, two and three dimensions. We are particularly interested in how it improves the conditioning of the least-squares problem given by \Cref{eq:ELMFBPINN_ls} and its subsequent solution accuracy.

This section is structured as follows. First, in \Cref{sec:problems-studied}, we describe the problems studied in detail. Next, in \Cref{sec:common-implementation}, we explain the common implementation details across all experiments. Then, in \Cref{sec:baselines} we describe the baseline models we compare to. Finally, in \Cref{sec:results} we describe our numerical results in detail.

\subsection{Problems studied}
\label{sec:problems-studied}

This section describes the BVPs and IVPs studied. For each problem, we carry out ablation studies of the \model{}. We assess how its accuracy and convergence changes with different $\sigma$ values in the RRQR filtering, depths of subdomain networks, and subdomain network activation functions. We also carry out strong and weak scaling studies, where we assess how the total number of subdomains, $J$ in \Cref{eq:elm_solution}, the degree of overlap between subdomains, and number of basis functions, $K$ in \Cref{eq:elm_solution}, affects performance, and how our method performs when the solution complexity increases.

\subsubsection{Harmonic oscillator in 1D}\label{sec:harmonic_oscilator_problem}

The first problem we study is the damped harmonic oscillator in one dimension. This models the displacement, $u(t)$, of a mass on a spring through time, and is given by
\begin{equation} \label{eq:harmonic_oscilator_problem}
\begin{split}
    m\frac{d^2 u}{d t^2} + \mu \frac{d u}{d t} + ku &= 0 \quad \text{in } \Omega = [0,1], \\
    u(0) &= 1, \\
    \frac{d u}{d t}(0) &= 0,
\end{split}
\end{equation}
where $m$ is the mass of the oscillator, $\mu$ is the coefficient of friction, and $k$ is the spring constant. We focus on solving the problem in the under-damped state, which occurs when
\begin{equation} \label{eq:under-damped}
\delta < \omega_0, \quad \text{where} \quad \delta = \dfrac{\mu}{2m}, \quad \omega_0 = \sqrt{\dfrac{k}{m}}.
\end{equation}
In this case, the exact solution is known and is given by
\begin{equation}
u(t) = e^{-\delta t} \left(2 A \cos(\phi + \omega t)\right), \quad \text{with} \quad \omega = \sqrt{\omega_0^2 - \delta^2},
\end{equation}
where $A$ and $\phi$ can be derived from the initial conditions. For all cases studied, we fix $\mu=4$, and $m=1$, and when performing weak scaling studies we increase the solution complexity by increasing $\omega_0$.

\subsubsection{Multi-scale Laplacian in 2D}\label{sec:laplace_problem}

The second problem we study is a multi-scale Laplacian problem in two dimensions, as defined in \cite{Dolean:MDD:2024}, given by
\begin{equation} \label{eq:laplace_problem}
\begin{split}
- \Delta u & = f \quad \text{in } \Omega = [0,1]^2, \\
u & = 0 \quad \text{on } \partial\Omega,
\end{split}
\end{equation}
with a source term given by
\begin{equation} \label{eq:multiscale_laplace_f}
f(x_1,x_2) = \frac{2}{n}\sum_{i=1}^n \big( \omega_i \pi \big)^2 \sin({\omega_i \pi x_1}) \, \sin({\omega_i \pi x_2}).
\end{equation}
Then, the exact solution is given by 
\begin{equation}
u(x_1,x_2) = \frac{1}{n}\sum_{i=1}^n \sin({\omega_i \pi x_1}) \, \sin({\omega_i \pi x_2}).
\end{equation}
In this study we fix $\omega_i = 2^{i}$. The solution to this problem is highly multi-scale, with the number of multi-scale components controlled by the parameter $n$. When performing weak scaling studies we increase the solution complexity by increasing $n$.

\paragraph{Hard boundary constraints}{
For this problem, we use a hard-constraining approach, as originally described by \cite{Lagaris1998}, to assert the boundary conditions in the ELM-FBPINN. Specifically, we assume that the solution has the following ansatz
\begin{equation}
\tilde{u}(x_1,x_2,\ba) = \tanh(\omega_n x_1)\tanh(\omega_n(1-x_1))\tanh(\omega_n x_2)\tanh(\omega_n(1-x_2))\,\hat{u}(x_1,x_2,\ba),
\end{equation}
with $\hat{u}$ as defined in \Cref{eq:elm_solution}. It is simple to show that this ansatz automatically obeys the boundary condition in \Cref{eq:laplace_problem}. This has the advantage that when training the ELM-FBPINN the boundary losses in \Cref{eq:ELMFBPINN_loss} can be removed. It is important to note that when using this ansatz \Cref{eq:ELMFBPINN_loss} remains quadratic in $\ba$, which still ensures we can form the least-squares system defined in \Cref{eq:ELMFBPINN_ls}.
}

\subsubsection{Time-varying wave equation in (2+1)D}\label{sec:wave_equation_problem}

The final problem studied is the time-varying two dimensional acoustic wave equation with Dirichlet boundary conditions, given by
\begin{equation}\label{eq:wave_equation_problem}
\begin{split}
\nabla^{2}u - \frac{1}{c^2}\frac{\partial^2 u}{\partial t^2}  &= 0 \quad \text{in } \Omega = [-1,1]^2 \times [0,1], \\
u & = 0 \quad \text{on } \partial\Omega_{\bx} \times [0,1], \\
u(\bx,t=0) &= e^{-\frac{1}{2}(\|\bx-\bmu\|/\lambda)^2},\\
\frac{\partial u}{\partial t} (\bx,t=0) &= 0,
\end{split}
\end{equation}
where $c \in \mathbb{R}^{1}$ is a constant wave speed, and $\bmu \in \mathbb{R}^{2}$ and $\lambda \in \mathbb{R}^{1}$ control the starting location and central frequency of a stationary Gaussian point source. For all problems studied we set $c=1$ and $\bmu=(0,0)$. 

With this setup the solution is a circular wavefront expanding from the center of the domain and reflecting from the boundaries. We are not aware of an exact solution, and therefore compare our results to a high-fidelity finite difference acoustic wave equation solver, described in detail in \ref{sec:appendix_fd}. When performing weak scaling studies we increase the solution complexity by decreasing $\lambda$ which increases the solution frequency.

\paragraph{Hard boundary constraints}{Similar to the multi-scale Laplacian problem in \Cref{sec:laplace_problem}, we use a hard boundary constraint when solving this problem, assuming the solution has the following ansatz
\begin{align}
\tilde{u}(x_1, x_2, t,\ba) &= \tanh(x_1 / d)\tanh((1-x_1)/ d)\tanh( x_2/ d)\tanh((1-x_2)/ d) \\
& \quad \times \left(  e^{-\frac{1}{2}\left((0.6\,t/\tau)^2 + \|\bx-\bmu\|/\lambda\right)^2} +  \tanh^2(t/\tau)\hat{u}(x_1, x_2, t,\ba) \right),
\end{align}
where $d=\lambda/2.5$ and $\tau=\lambda/(2.5\,c)$. This ansatz is designed so that all boundary conditions in \Cref{eq:wave_equation_problem} are automatically obeyed. It is also designed such that at times $t \gg \tau$, the initial Gaussian point source term is negligibly zero. This is done so that the ELM-FBPINN does not need to correct for this part of the ansatz once the point source has expanded away from its starting location.
}

\subsection{Common implementation details}
\label{sec:common-implementation}

The \modelss{} tested share many common implementation details across the problems studied which are described below.

\paragraph{Domain decomposition}{
All problems are defined on rectangular domains such that $\Omega=\prod_{i=1}^d [a_i,b_i]$, where $a_i$ and $b_i$ are the edges of the domain along each input dimension. For simplicity, all models tested use a regular overlapping rectangular domain decomposition, defined as $D=\{\Omega_{\bj}\}$ with $ \bj \in \{1,\ldots,S\}^d$, where $d$ is the number of input dimensions, $S$ is the number of subdomains along each input dimension, $J=S^d$ is the total number of subdomains, and
\begin{equation}\label{eq:dd}
\Omega_{\bj} =
\prod_{i=1}^d \Omega_{j_i},
\quad \text{where} \quad \Omega_{j_i} =
\left[
a_i + (b_i-a_i)\frac{(j_i-1)-\delta/2}{S-1},
a_i + (b_i-a_i)\frac{(j_i-1)+\delta/2}{S-1}
\right],
\end{equation}
where $\delta$ is defined as the \textit{overlap ratio}. Note that an overlap ratio of less than 1 means that the subdomains are no longer overlapping. Example domain decompositions used in 1D and 2D are shown in \Cref{fig:harmonic-oscillator-solution-basis} and \Cref{fig:laplace-solution-condition}.

As discussed in \Cref{sec:ELM-FBPINN-background}, window functions with local support are used to confine each ELM-FBPINN subdomain network to its subdomain. All models tested use the following window function for each subdomain
\begin{equation} \label{eq:used_window_function_nd}
\omega_{\bj} = \frac{\hat \omega_{\bj}}{\sum_{\bj} \hat\omega_{\bj}},
\quad \text{where} \quad \hat \omega_{\bj}(\mathbf{x}) = \prod_{i=1}^d [1+\cos(\pi(x_i-\mu_{j_i}/\sigma_{j_i})]^2,
\end{equation}
such that the window functions form a partition of unity, where $\mu_{j_i}$ and $\sigma_{j_i}$ are the center and half-width of each subdomain along each input dimension.

}
\paragraph{Network architecture}{
All models tested use a fully-connected neural network for each subdomain network, with the same size network in each subdomain. All hidden subdomain network parameters are randomly initialized using uniform LeCun initialization. Furthermore, the $\bx$ inputs to each subdomain network are normalized to the range $[-1,1]$ along each dimension over their individual subdomains.

}
\paragraph{Interior collocation points}{
For all models tested, we choose the number of interior collocation points for the PDE residual term in \Cref{eq:ELMFBPINN_loss} such that the global matrix $\mat{M}$ is always well-determined (square or otherwise tall). Specifically, we sample interior collocation points on a regular grid over the entire domain $\Omega$ with $S \cdot \text{ceil}(K^{\frac{1}{d}})$ collocation points along each input dimension. The total number of interior collocation points (which each contribute a row in $\mat{M}$) is therefore $\left[S \cdot \text{ceil}(K^{\frac{1}{d}})\right]^d \ge S^d K = JK$, where we note that $JK$ is the total number of columns in the global $\mat{M}$ matrix.

}
\paragraph{Least-squares solver}{
All models tested use \Cref{alg:RRQR_solver} with the iterative LSQR solver \cite{Paige1982} to solve the global least-squares problem (\Cref{eq:ELMFBPINN_ls}). This solver operates directly on the least-squares problem and avoids potential pitfalls associated with normal-equation-based solution approaches. We purposefully run the solver over 10,000 iterations without using a stopping condition, so that the full convergence behavior of each model can be displayed. In practice, the majority of models tested converge with much less iterations and a stopping condition can be used to significantly accelerate training times. When reporting the final error of a solve, we report the minimum error observed over all iteration steps.

}
\paragraph{Performance evaluation}{
All models are evaluated using the normalized L1 test accuracy, given by $e_{L_1}(\ba)=\frac{1}{M}\sum_i^M \| \hat{u}(\bx_i,\ba)-u(\mathbf{x}_i) \| / \sigma$, where $M$ is the number of test points and $\sigma$ is the standard deviation of the set of true solutions $\{u(\mathbf{x}_i)\}^M_i$. Test points are sampled on a regular grid over the entire domain $\Omega$. For robustness, all models are trained five times using different random seeds. Test accuracies are reported as the median value \(\pm\) the range (max -- min), while training times are reported as the mean \(\pm\) one standard deviation.

}
\paragraph{Software and hardware implementation}{
We implement \modelss{} by building upon the highly optimized FBPINN JAX library available here: \href{https://github.com/benmoseley/FBPINNs}{github.com/benmoseley/FBPINNs}, described in more detail in \cite{Moseley2023, Dolean:MDD:2024, jax2018github}. We define sparse linear solvers using the SciPy library \cite{2020SciPy-NMeth}, and all models are trained on a single AMD EPYC 7742 CPU core with 50~GB RAM.

}

\subsection{Baselines}\label{sec:baselines}

\newcommand{\modelsNoPre}{ELM-FBPINNs without preconditioning}
\newcommand{\modelLSQR}{ELM-FBPINN without preconditioning with LSQR}
\newcommand{\modelAS}{ELM-FBPINN with AS preconditioning}

We compare our \modelss{} to the following baselines:
\begin{enumerate}
  \item PINNs \cite{Lagaris1998, Raissi2019}, optimizing all parameters of a single global neural network to minimize \Cref{eq:ELMFBPINN_loss} with gradient descent using the Adam optimizer and a learning rate of \SI{1e-3}{},
  \item FBPINNs \cite{Moseley2023}, optimizing all subdomain network parameters in \Cref{eq:elm_solution} (rather than freezing hidden parameters) to minimize \Cref{eq:ELMFBPINN_loss} with gradient descent using the Adam optimizer and a learning rate of \SI{1e-3}{},
  \item ELM-FBPINNs without preconditioning, solving the global least-squares problem (\Cref{eq:ELMFBPINN_ls}) directly with sparse LSQR,
  \item ELM-FBPINNs without preconditioning, solving the global least-squares problem (\Cref{eq:ELMFBPINN_ls}) via its normal equations with the conjugate gradient (CG) method,
  \item ELM-FBPINNs with additive Schwarz (AS) preconditioning proposed by \cite{Shang_Heinlein_Mishra_Wang_2025}, solving the global least-squares problem (\Cref{eq:ELMFBPINN_ls}) via its normal equations with AS preconditioning without PCA filtering and the CG method.
\end{enumerate}
For fairness, when comparing these models to \modelss{}, all relevant hyperparameters are kept the same, including problem definition, domain decomposition, subdomain network architecture, parameter initialization strategy, random seed, collocation point definition, loss function, and testing point definition. For the PINN and FBPINN baselines we optimize all network parameters. This turns \Cref{eq:ELMFBPINN_loss} into a nonlinear optimization problem and we therefore use gradient descent with 10,000 steps to optimize them. Furthermore, given PINNs do not use a domain decomposition, a larger global neural network with more hidden layers and hidden units is used for these models, and their inputs are normalized to the range $[-1,1]$ along each dimension over the global domain.

\subsection{Results}
\label{sec:results}


In this section we discuss the performance of \modelss{} on the problems described in \Cref{sec:problems-studied}.

\subsubsection{Harmonic oscillator in 1D}\label{sec:harmonic_oscilator_results}

\paragraph{Baseline model}{\label{sec:harmonic_oscillator_baseline} We first study the harmonic oscillator problem with $\omega_0=60$. We train a \model{} with $J=S=20$, $\delta=2.9$, $K=8$, a subdomain network depth of $h=1$, and $\tanh$ activation function, using $\sigma=\SI{1e-8}{}$ during $\sigma$-RRQR filtering, and test the model with $50\times S=1000$ collocation points. The true solution, domain decomposition, subdomain network basis functions and model solution are shown in \Cref{fig:harmonic-oscillator-solution-basis}. We find that the model is able to accurately predict the exact solution, with a normalized L1 test accuracy of $1.0\pm\SI{1.2e-04}{}$. During $\sigma$-RRQR filtering, we find that 8\% of columns in $\mat{M}$ are discarded, and the condition number of $\mat{M}$ decreases from \SI{2.9e+16}{} to \SI{3.0e+11}{}. The condition number of the preconditioned matrix $\widehat{\mat{M}} \mat{S}^{-1}$ further drops to \SI{5.2e+05}{}. This shows that both our filtering and preconditioning steps significantly improve the conditioning of the least-squares problem. The global matrix $\mat{M}$ and preconditioned matrix $\widehat{\mat{M}} \mat{S}^{-1}$, and their SVD spectra, are plotted in \Cref{fig:harmonic-oscillator-condition}, which shows the expected block-sparsity pattern of these matrices, and a reduction in dynamic range of their SVD spectra after preconditioning. The convergence of the LSQR solver is shown in \Cref{fig:harmonic-oscillator-convergence}. We find that the model converges within 1000 iterations, whilst the total time taken to carry out the full 10,000 iterations is $4.0\pm\SI{0.1}{s}$.
}

\begin{figure}[!t]
\centering
\includegraphics[width=\textwidth]{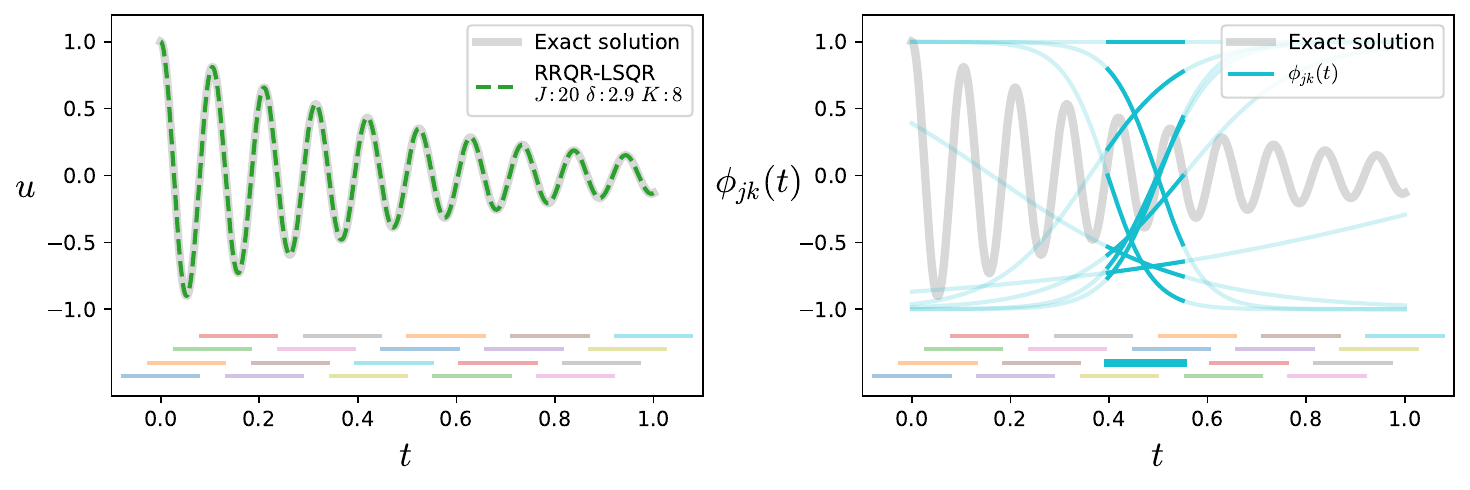}
\caption{Exact solution (left), domain decomposition (both), model basis functions for the 10\textsuperscript{th} subdomain (right), and model solution (left) when using the \model{} to solve the harmonic oscillator in 1D.
\label{fig:harmonic-oscillator-solution-basis}}
\end{figure}

\begin{figure}[!t]
\centering
\includegraphics[width=\textwidth]{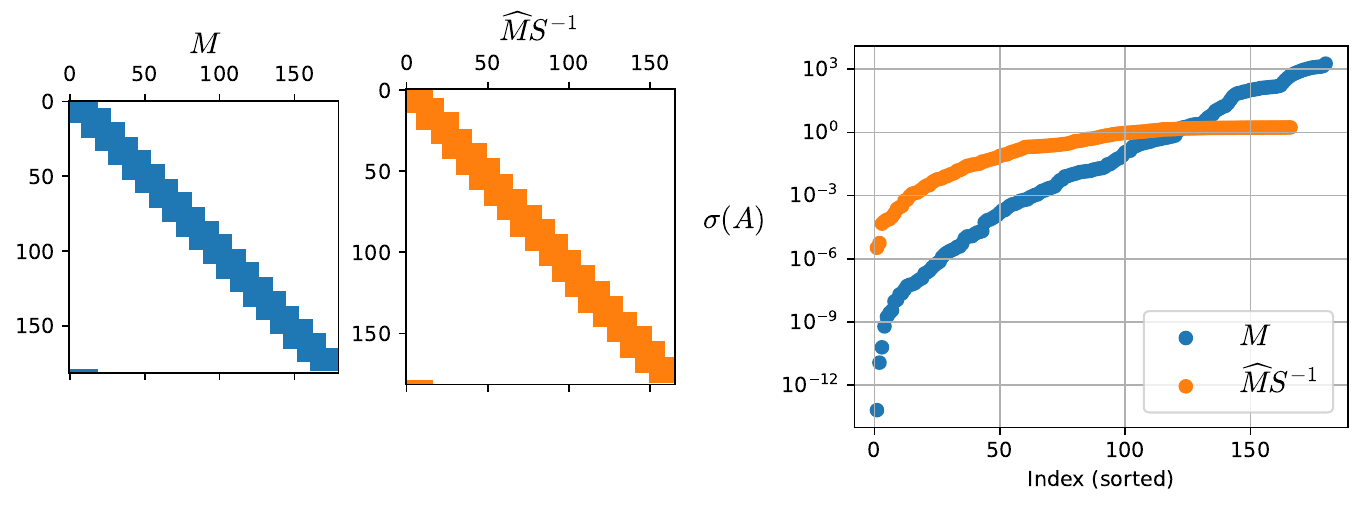}
\caption{Sparsity plot of $\mat{M}$ and $\widehat{\mat{M}} \mat{S}^{-1}$ (left) and their SVD spectra (right), when using the \model{} to solve the harmonic oscillator in 1D.
\label{fig:harmonic-oscillator-condition}}
\end{figure}

\paragraph{Baseline comparison}{Next, we compare our baseline \model{} above to our baseline models. A comparison of the relevant condition numbers, test accuracy and total training time is shown in \Cref{tab:harmonic-oscillator}. Our model has the highest test accuracy, with the \modelAS{} and \modelsNoPre{} achieving a slightly lower test accuracy. The FBPINN and PINN baselines perform significantly worse than the models with linear solvers, being one and three orders of magnitude less accurate respectively, possibly because of the non-convexity of their loss function. All models with linear solvers have similar training times, whilst the FBPINN and PINN baselines are and nearly an order of magnitude slower to train due to the higher computational cost of gradient descent. The convergence curves of all linear solver models are compared in \Cref{fig:harmonic-oscillator-convergence}. We find that the \modelsNoPre{} take the full 10,000 iteration steps to converge, whilst both our model and the \modelAS{} converge within 1000 steps. This shows that preconditioning is essential for enabling the solver to converge quickly. We notice that whilst the \modelAS{} converges quickly, it quickly becomes unstable at later iteration steps. Thus, it appears in practice a good stopping condition is required to use this model.
}

\begin{table}[h]
\setlength{\tabcolsep}{2pt}
\centering
\resizebox{\textwidth}{!}{
\begin{tabular}{ccccccccccc}
\toprule
Network & $(h,K)$ & $\kappa(M)$ & $\kappa(\widehat{M})$ & $\kappa(\widehat{M} S^{-1})$ & $\kappa(M^T M)$ & $\kappa(A^{-1}_{AS} A)$ & Optimiser & $\sigma$ ($\phi_{jk}$ drop \%) & $e_{L_1}$ & Time (s) \\
\midrule
FCN & (2, 64) & N/A & N/A & N/A & N/A & N/A & PINN-Adam & N/A (N/A) & 4.8$\pm$0.6e-01 & 1.4$\pm$0.0e+01 \\
FCN & (1, 8) & N/A & N/A & N/A & N/A & N/A & FBPINN-Adam & N/A (N/A) & 3.5$\pm$1.5e-03 & 8.9$\pm$0.1 \\
ELM & (1, 8) & 2.9e+16 & N/A & N/A & N/A & N/A & LSQR & N/A (N/A) & 2.9$\pm$1.8e-04 & 3.4$\pm$0.0 \\
ELM & (1, 8) & 2.9e+16 & N/A & N/A & 9.9e+18 & N/A & CG & N/A (N/A) & 6.0$\pm$9.2e-04 & 2.9$\pm$0.0 \\
ELM & (1, 8) & 2.9e+16 & N/A & N/A & 9.9e+18 & 3.2e+12 & AS-CG & N/A (N/A) & 2.7$\pm$7.5e-04 & 3.6$\pm$0.0 \\
ELM & (1, 8) & 2.9e+16 & 3.0e+11 & 5.2e+05 & N/A & N/A & RRQR-LSQR & 1e-08 (8\%) & 1.0$\pm$1.2e-04 & 4.0$\pm$0.1 \\
ELM & (1, 8) & 2.9e+16 & 2.8e+12 & 3.7e+06 & N/A & N/A & RRQR-LSQR & 1e-10 (2\%) & 5.9$\pm$9.0e-05 & 4.1$\pm$0.1 \\
ELM & (1, 8) & 2.9e+16 & 1.6e+08 & 1.4e+04 & N/A & N/A & RRQR-LSQR & 1e-06 (21\%) & 1.1$\pm$3.1e-04 & 4.1$\pm$0.1 \\
ELM & (1, 8) & 2.9e+16 & 2.7e+05 & 2.4e+02 & N/A & N/A & RRQR-LSQR & 1e-04 (43\%) & 1.0$\pm$0.7e-03 & 4.0$\pm$0.1 \\
ELM & (1, 8) & 2.9e+16 & 2.6e+03 & 5.9e+01 & N/A & N/A & RRQR-LSQR & 1e-02 (67\%) & 1.3$\pm$1.1e-02 & 4.2$\pm$0.1 \\
ELM-$\sigma$ & (1, 8) & 4.9e+17 & 1.8e+10 & 5.0e+03 & N/A & N/A & RRQR-LSQR & 1e-08 (22\%) & 2.0$\pm$2.3e-04 & 4.2$\pm$0.3 \\
ELM-$\sin$ & (1, 8) & 6.4e+16 & 1.1e+10 & 3.2e+03 & N/A & N/A & RRQR-LSQR & 1e-08 (23\%) & 2.0$\pm$2.1e-04 & 4.1$\pm$0.1 \\
ELM & (3, 8) & 2.4e+15 & 4.5e+11 & 6.5e+05 & N/A & N/A & RRQR-LSQR & 1e-08 (3\%) & 4.9$\pm$9.4e-05 & 4.7$\pm$0.1 \\
ELM & (5, 8) & 7.6e+15 & 2.6e+11 & 1.7e+05 & N/A & N/A & RRQR-LSQR & 1e-08 (12\%) & 5.7$\pm$7.9e-05 & 5.3$\pm$0.1 \\
\bottomrule
\end{tabular}

}
\caption{Comparison of test error, total training time, and conditioning numbers for different models solving the harmonic oscillator in 1D.}
\label{tab:harmonic-oscillator}
\end{table}

\begin{figure}[!t]
\centering
\includegraphics[width=0.75\textwidth]{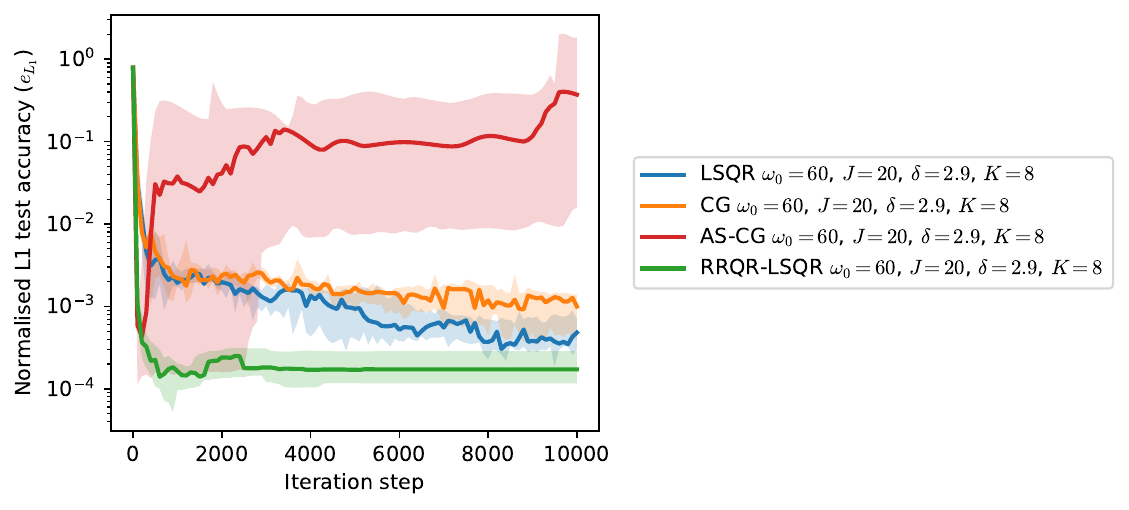}
\caption{Convergence curves of the \model{} and our baseline models when solving the harmonic oscillator in 1D.
\label{fig:harmonic-oscillator-convergence}}
\end{figure}

\paragraph{Ablation tests}{
Next, we perform ablation tests on our baseline \model{}. Specifically, we test its performance whilst varying $\sigma$ between $\SI{1e-10}{},\, \SI{1e-6}{},\, \SI{1e-4}{}, \text{ and } \SI{1e-2}{}$, using sigmoid and sinusoidal activation functions, and with a subdomain network depth of $h=3 \text{ and } 5$. The test accuracy and total training time of these models is reported in \Cref{tab:harmonic-oscillator}. As expected, the number of columns dropped by $\sigma$-RRQR filtering increases as $\sigma$ increases, with 67\% of columns dropped for $\sigma=\SI{1e-2}{}$. Also as expected, the condition numbers of $\widehat{\mat{M}}$ and $\widehat{\mat{M}} \mat{S}^{-1}$ decrease as $\sigma$ increases, as more numerically dependent columns are dropped during filtering. However, the test accuracy decreases as $\sigma$ increases, possibly because the model has less representational capacity as columns (basis functions) are dropped. Thus it appears important to balance good conditioning of the system against test time accuracy. We find that using sigmoid and sinusoidal activation functions does not improve test accuracy, whilst increasing $h$ does improve test accuracy and reduces the condition number of $\mat{M}$, possibly because increasing the depth of the subdomain networks allows the basis functions to be more flexible and therefore less linearly dependent.
}

\paragraph{Strong scaling tests}{
Next we perform strong scaling tests, where the general principle is to test \textit{fixing the problem complexity and increasing the model capacity}. In the ideal case, the model will improve its test accuracy in proportion with model capacity. For this problem, we fix the problem complexity by keeping $\omega_0=60$, as tested above. We then increase the model capacity in two different ways; first, by varying $K$ between $4,8,16,32,\text{ and }64$, and second by fixing $K=8$, but simultaneously varying $S$ between $10, 20, 40, 80,\text{ and }160$ and $\delta$ between $1.45, 2.9, 5.8, 11.6,\text{ and }23.2$. Scaling $S$ and $\delta$ in this way keeps the subdomain width fixed at $0.153$ for all tests, but adds more subdomains with higher overlap. All other hyperparameters are kept the same as our baseline model above. Plots of the convergence curves and test accuracy versus total training time of the \model{}, \modelLSQR{}, and \modelAS{} for each scaling test are shown in \Cref{fig:harmonic-oscillator-strong-p-scaling} and \Cref{fig:harmonic-oscillator-strong-wm-scaling}. For all solvers, we find that the test accuracy improves when increasing $K$, showing good scaling behavior. However, the \modelLSQR{} and \modelAS{} have significantly worse test accuracy than our model for large values of $K$. This is likely because the condition number of $\mat{M}$ increases significantly with $K$, and our model is the only model which filters basis functions (i.e. lowers $K$) before solving the system. Furthermore, our model is an order of magnitude faster than than the \modelLSQR{}, and two orders of magnitude faster than the \modelAS{}. This is both because our model drops columns before solving the system, which reduces the size of $\mat{M}$, and because computation of the $A^{-1}_{AS}$ preconditioner is significantly more expensive in memory and time than our $\widehat{\mat{M}} \mat{S}^{-1}$ preconditioner. For all solvers, we find that the test accuracy improves when simultaneously increasing $S$ and $\delta$, where in this case the \modelAS{} achieves a slightly higher test accuracy than our model for the final $S=160, \delta=23.2$ case, but does takes significantly longer to do so.
}

\begin{figure}[!t]
\centering
\includegraphics[width=\textwidth]{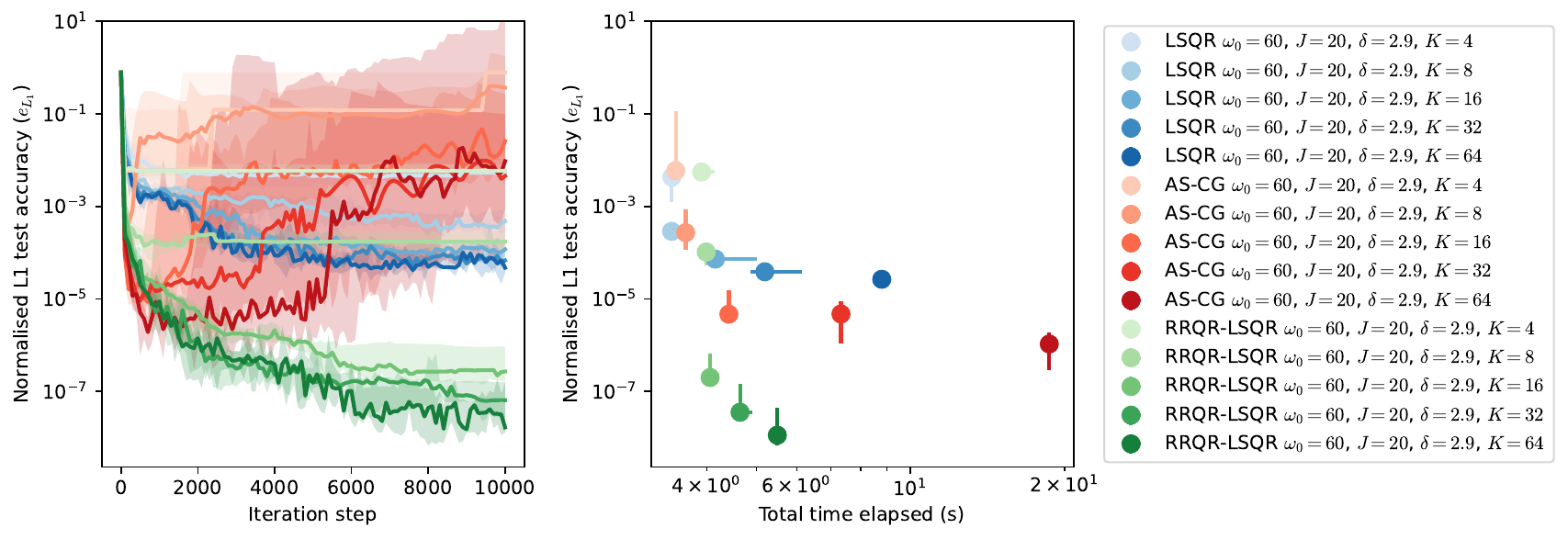}
\caption{Strong scaling test for the harmonic oscillator in 1D, fixing the solution frequency, $\omega_0$, and varying the number of basis functions, $K$.
\label{fig:harmonic-oscillator-strong-p-scaling}}
\end{figure}

\begin{figure}[!t]
\centering
\includegraphics[width=\textwidth]{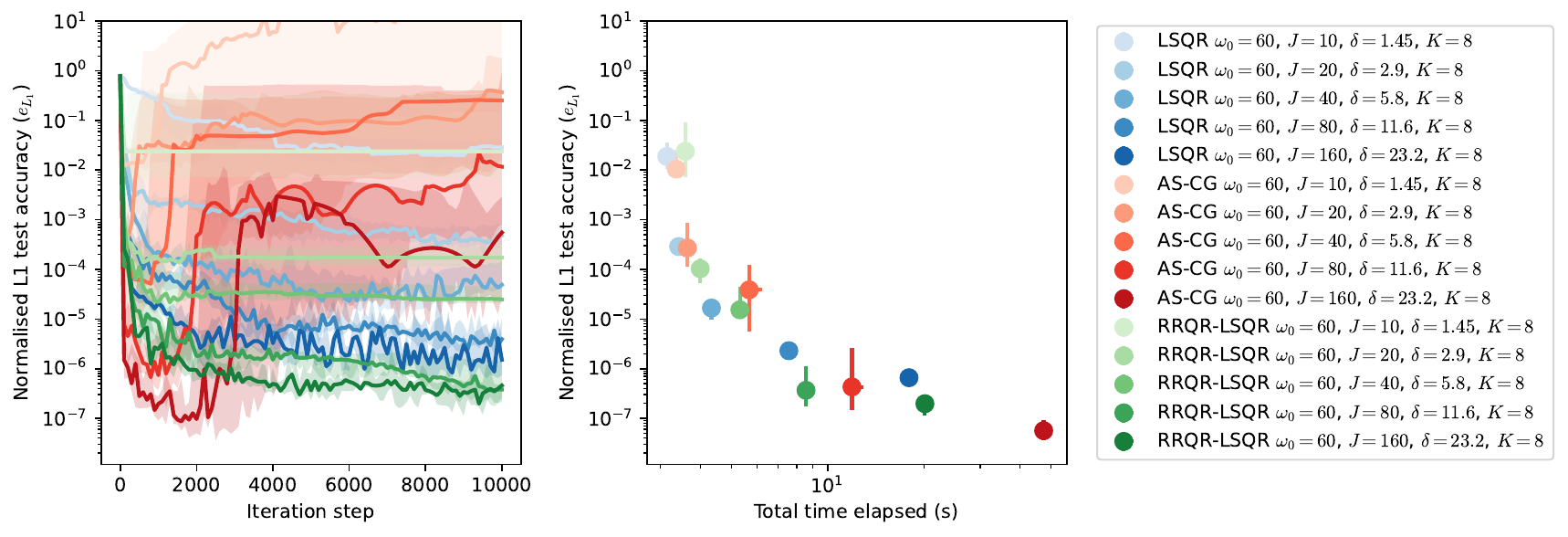}
\caption{Strong scaling test for the harmonic oscillator in 1D, fixing the solution frequency, $\omega_0$, and simultaneously varying the number of subdomains, $J$, and the overlap between subdomains, $\delta$.
\label{fig:harmonic-oscillator-strong-wm-scaling}}
\end{figure}

\paragraph{Weak scaling test}{Related to strong scaling tests, we also perform weak scaling tests, where the general principle is to test \textit{increasing the problem complexity whilst proportionally increasing the model capacity}. In the ideal case, test accuracy will stay constant as the problem complexity and model capacity is increased. For this problem, we increase the problem complexity by varying $\omega_0$ between $30, 60, 120, 240, \text{ and }580$, whilst simultaneously increasing $S$ between $10, 20, 40, 80,\text{ and }160$. All other hyperparameters are kept the same as our baseline model above. Plots of the convergence curves and test accuracy versus total training time of the \model{}, \modelLSQR{}, and \modelAS{} for each case are shown in \Cref{fig:harmonic-oscillator-weak-scaling}, and the exact solution and \model{} solution are plotted for different $\omega_0$ values in \Cref{fig:harmonic-oscillator-weak-solution}. We find that our model and the \modelAS{} have good scaling behavior, with test accuracy roughly staying constant as $\omega_0$ increases, whilst the \modelLSQR{} has poor scaling behavior, being unable to accurately model the $\omega_0=240 \text{ and }580$ solutions with the solver converging extremely slowly. This presents further evidence that preconditioning is essential especially for larger numbers of subdomains.
}

\begin{figure}[!t]
\centering
\includegraphics[width=\textwidth]{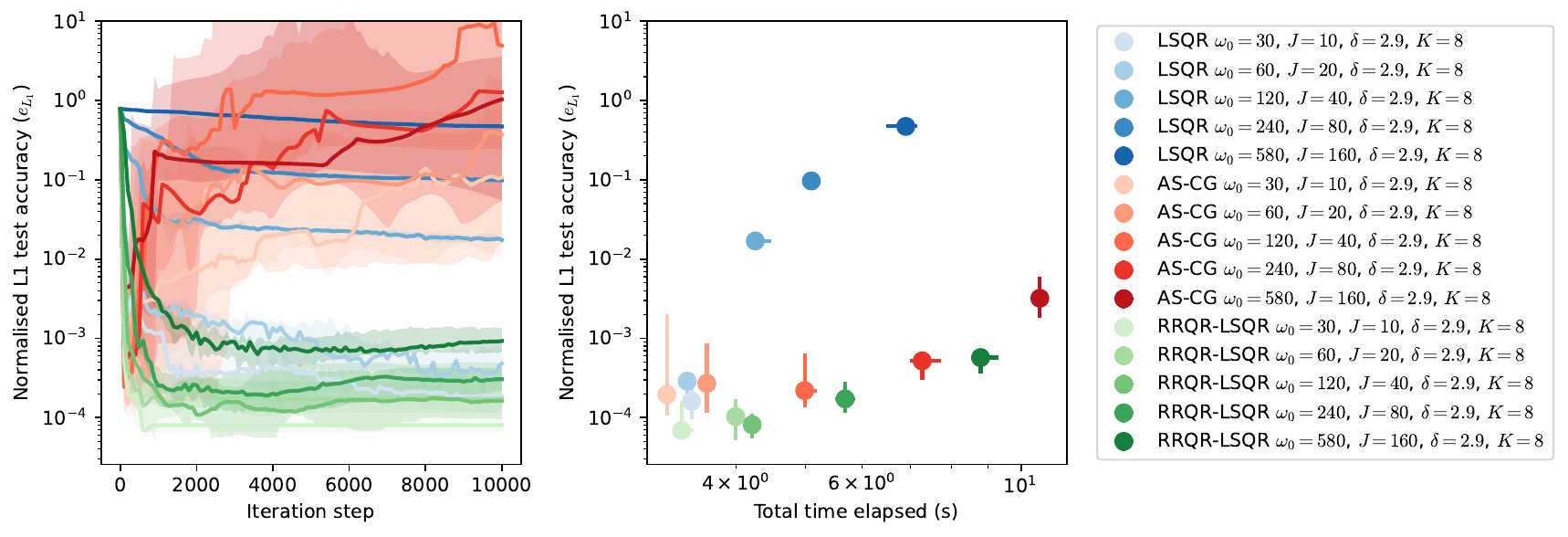}
\caption{Weak scaling test for the harmonic oscillator in 1D, increasing the solution frequency, $\omega_0$, whilst simultaneously increasing the number of subdomains, $J$.
\label{fig:harmonic-oscillator-weak-scaling}}
\end{figure}

\begin{figure}[!t]
\centering
\includegraphics[width=\textwidth]{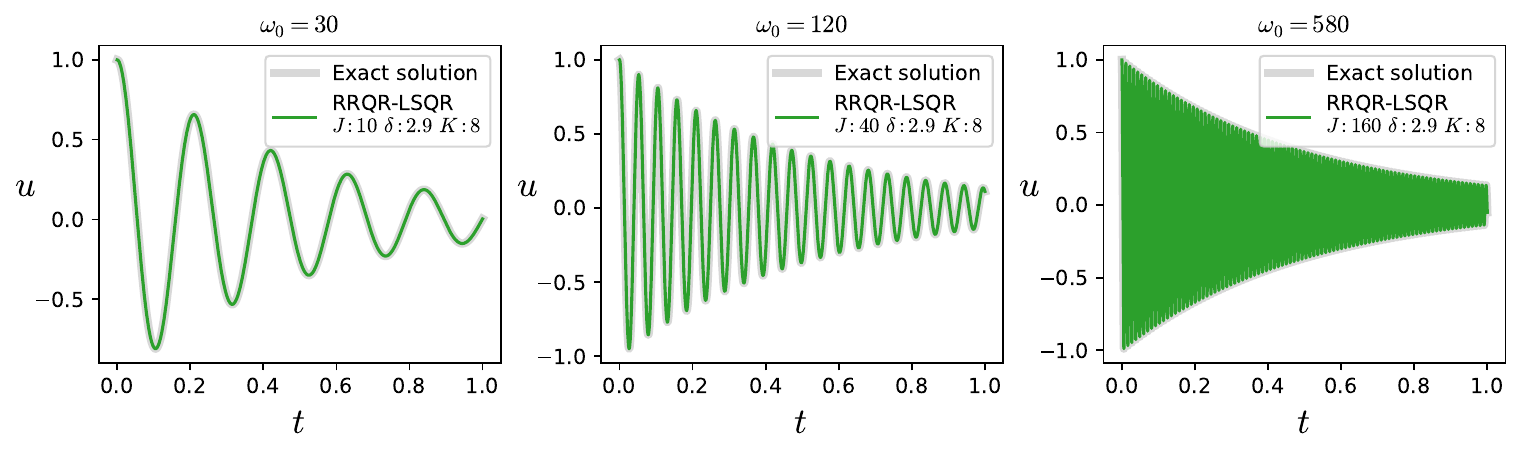}
\caption{Comparison of exact solution and \model{} solution for three of the cases in the weak scaling test for the harmonic oscillator in 1D.
\label{fig:harmonic-oscillator-weak-solution}}
\end{figure}

\subsubsection{Multi-scale Laplacian in 2D}\label{sec:laplace_results}

\paragraph{Baseline model}{Next, we move on to studying the multi-scale Laplacian problem in 2D. As above, we first train a baseline \model{}, studying the problem with $n=3$, using $S=16$, $\delta=2.9$, $K=16$, a subdomain network depth of $h=1$, $\tanh$ activation function, $\sigma=\SI{1e-8}{}$, and $6\times S=96$ collocation points along each input dimension for testing. The true solution, domain decomposition, and model solution are shown in \Cref{fig:laplace-solution-condition}. Similar to the 1D harmonic oscillator problem, we find that the model is able to accurately predict the exact solution, with a normalized L1 test accuracy of $1.9\pm\SI{1.3e-03}{}$. Interesting, we find that no columns are dropped during $\sigma$-RRQR filtering, likely because the condition number of $\mat{M}$ is already relatively low at \SI{1.4e+9}{}. After preconditioning the condition number of $\widehat{\mat{M}} \mat{S}^{-1}$ further drops to \SI{2.8e+04}{}, and the SVD spectra of $\mat{M}$ and $\widehat{\mat{M}} \mat{S}^{-1}$ are plotted in \Cref{fig:laplace-solution-condition}. The convergence of the LSQR solver is shown in \Cref{fig:laplace-convergence}. We find that the model converges within a few hundred steps, whilst the total time taken to carry out the full 10,000 iterations is $3.1\pm\SI{0.1e01}{s}$.
}

\begin{figure}[!t]
\centering
\includegraphics[width=0.75\textwidth]{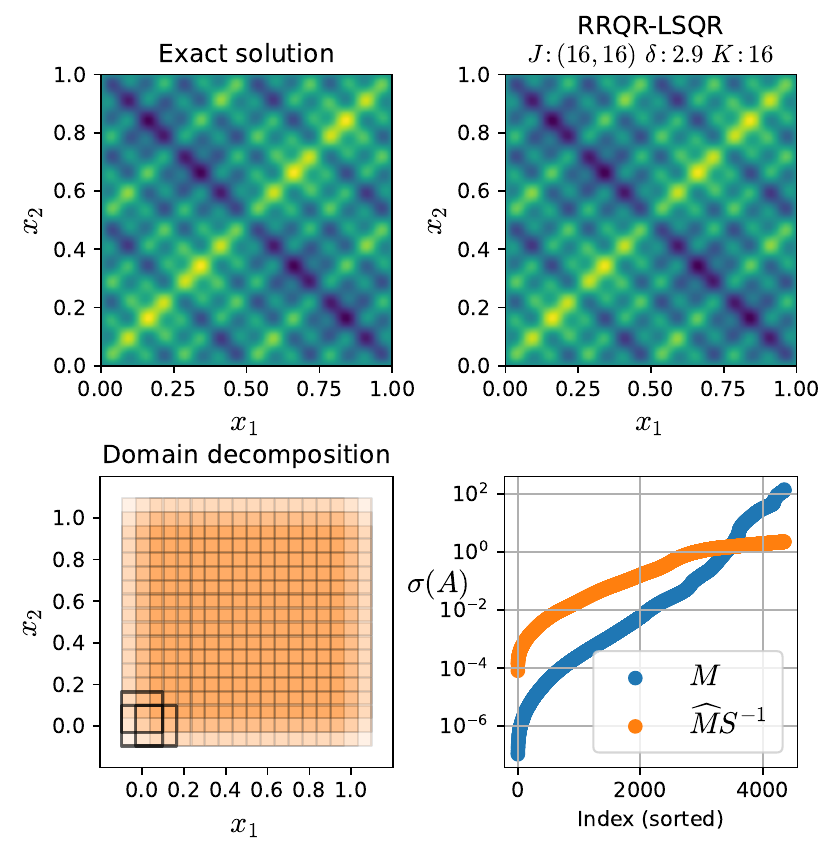}
\caption{Exact solution (top left), domain decomposition (bottom left), model solution (top right), and SVD spectra of $\mat{M}$ and $\widehat{\mat{M}} \mat{S}^{-1}$ (bottom right) when using the \model{} to solve the multi-scale Laplacian in 2D.
\label{fig:laplace-solution-condition}}
\end{figure}

\paragraph{Baseline comparison}{We compare our baseline \model{} above to our baseline models. A comparison of the relevant condition numbers, test accuracy and total training time is shown in \Cref{tab:laplace}. Similar results can be drawn to the harmonic oscillator test above; our model has the highest test accuracy, with the \modelAS{} and \modelsNoPre{} achieving a slightly lower test accuracy with comparable training times, whilst the FBPINN and PINN baselines perform significantly worse than our model in both accuracy and training time. The convergence curves of all linear solver models are compared in \Cref{fig:laplace-convergence}. Similar to the harmonic oscillator test, we find that the \modelsNoPre{} take the full 10,000 iteration steps to converge, whilst both our model and the \modelAS{} converge within a few hundred steps, again showing that preconditioning is essential for enabling the solver to converge quickly. Similar instabilities in the \modelAS{} at later iteration steps are observed.
}

\begin{table}[h]
\setlength{\tabcolsep}{2pt}
\centering
\resizebox{\textwidth}{!}{
\begin{tabular}{ccccccccccc}
\toprule
Network & $(h,K)$ & $\kappa(M)$ & $\kappa(\widehat{M})$ & $\kappa(\widehat{M} S^{-1})$ & $\kappa(M^T M)$ & $\kappa(A^{-1}_{AS} A)$ & Optimiser & $\sigma$ ($\phi_{jk}$ drop \%) & $e_{L_1}$ & Time (s) \\
\midrule
FCN & (2, 128) & N/A & N/A & N/A & N/A & N/A & PINN-Adam & N/A (N/A) & 7.9$\pm$10.8e-02 & 3.2$\pm$0.0e+03 \\
FCN & (1, 16) & N/A & N/A & N/A & N/A & N/A & FBPINN-Adam & N/A (N/A) & 5.0$\pm$2.6e-03 & 1.5$\pm$0.0e+03 \\
ELM & (1, 16) & 1.4e+09 & N/A & N/A & N/A & N/A & LSQR & N/A (N/A) & 3.9$\pm$0.5e-03 & 3.1$\pm$0.1e+01 \\
ELM & (1, 16) & 1.4e+09 & N/A & N/A & 8.3e+18 & N/A & CG & N/A (N/A) & 4.2$\pm$1.1e-03 & 3.0$\pm$0.1e+01 \\
ELM & (1, 16) & 1.4e+09 & N/A & N/A & 8.3e+18 & 4.0e+15 & AS-CG & N/A (N/A) & 1.3$\pm$0.9e-02 & 1.1$\pm$0.0e+02 \\
ELM & (1, 16) & 1.4e+09 & 1.4e+09 & 2.8e+04 & N/A & N/A & RRQR-LSQR & 1e-08 (0\%) & 1.9$\pm$1.3e-03 & 3.1$\pm$0.1e+01 \\
ELM & (1, 16) & 1.4e+09 & 1.4e+09 & 2.8e+04 & N/A & N/A & RRQR-LSQR & 1e-10 (0\%) & 1.9$\pm$1.3e-03 & 3.1$\pm$0.1e+01 \\
ELM & (1, 16) & 1.4e+09 & 9.3e+08 & 2.2e+04 & N/A & N/A & RRQR-LSQR & 1e-06 (0\%) & 1.9$\pm$1.5e-03 & 3.7$\pm$0.1e+01 \\
ELM & (1, 16) & 1.4e+09 & 3.8e+06 & 1.1e+03 & N/A & N/A & RRQR-LSQR & 1e-04 (21\%) & 2.7$\pm$2.2e-03 & 2.8$\pm$0.1e+01 \\
ELM & (1, 16) & 1.4e+09 & 6.8e+03 & 1.1e+02 & N/A & N/A & RRQR-LSQR & 1e-02 (65\%) & 1.3$\pm$0.6e-02 & 1.7$\pm$0.0e+01 \\
ELM-$\sigma$ & (1, 16) & 1.4e+11 & 7.6e+10 & 2.4e+04 & N/A & N/A & RRQR-LSQR & 1e-08 (0\%) & 2.1$\pm$0.9e-03 & 3.5$\pm$0.7e+01 \\
ELM-$\sin$ & (1, 16) & 1.5e+10 & 1.5e+10 & 2.2e+04 & N/A & N/A & RRQR-LSQR & 1e-08 (0\%) & 2.2$\pm$0.6e-03 & 2.8$\pm$0.1e+01 \\
ELM & (3, 16) & 1.1e+09 & 1.1e+09 & 1.7e+04 & N/A & N/A & RRQR-LSQR & 1e-08 (0\%) & 2.8$\pm$2.1e-03 & 4.8$\pm$0.1e+01 \\
ELM & (5, 16) & 1.2e+10 & 1.2e+10 & 1.9e+04 & N/A & N/A & RRQR-LSQR & 1e-08 (0\%) & 2.7$\pm$2.8e-03 & 6.2$\pm$0.1e+01 \\
\bottomrule
\end{tabular}

}
\caption{Comparison of test error, total training time, and conditioning numbers for different models solving the multi-scale Laplacian in 2D.}
\label{tab:laplace}
\end{table}

\begin{figure}[!t]
\centering
\includegraphics[width=0.75\textwidth]{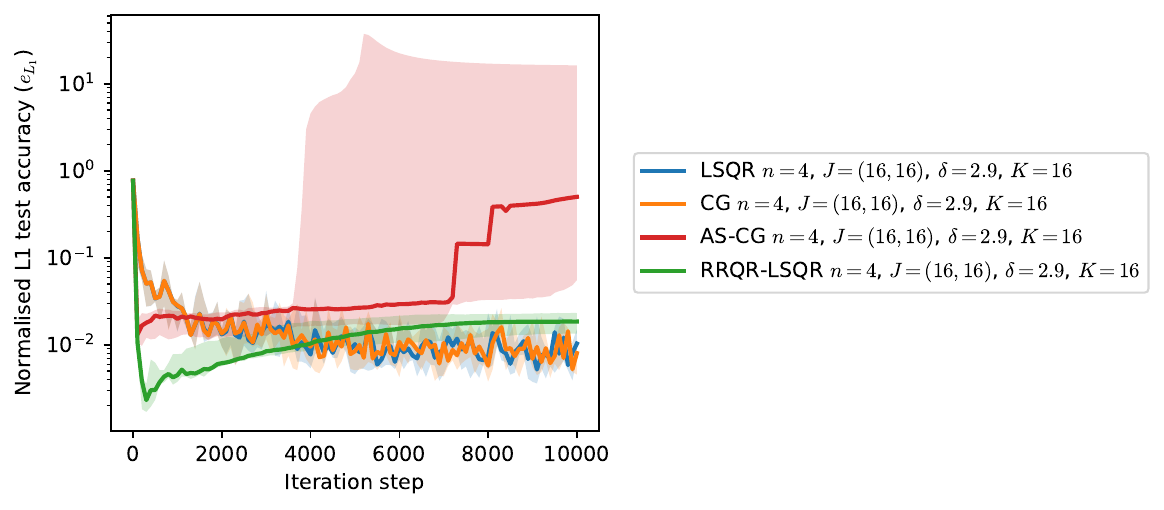}
\caption{Convergence curves of the \model{} and our baseline models when solving the multi-scale Laplacian in 2D.
\label{fig:laplace-convergence}}
\end{figure}

\paragraph{Ablation tests}{
We perform ablation tests on our baseline \model{}. Specifically, we test its performance using the same ablations studied for the harmonic oscillator problem. The test accuracy and total training time of these models is reported in \Cref{tab:laplace}. The same conclusions as the harmonic oscillator ablation studies can be drawn, except that increasing the subdomain network depth does not improve accuracy in this case.
}

\paragraph{Strong scaling tests}{
We perform a strong scaling test in a similar fashion to the harmonic oscillator test, fixing the problem complexity by keeping $n=3$, and increasing the model capacity in two different ways; first, by varying $K$ between $4,8,16,32,\text{ and }64$, and second by fixing $K=16$, but simultaneously varying $S$ between $8, 16, \text{ and }32$ and $\delta$ between $1.45, 2.9, \text{ and }5.8$. Plots of the convergence curves and test accuracy versus total training time of the \model{}, \modelLSQR{}, and \modelAS{} for each scaling test are shown in \Cref{fig:laplace-strong-p-scaling} and \Cref{fig:laplace-strong-wm-scaling}. For all solvers, we find that the test accuracy improves when increasing $K$, showing good scaling behavior, similar to the harmonic oscillator problem. However, for the largest $K$ value, the \modelLSQR{} is an order of magnitude less accurate than the other solvers, suggesting that good preconditioning is necessary with higher $K$ values. The \modelAS{} and our model perform with comparable accuracy, but, similar to the harmonic oscillator problem, the \modelAS{} is approximately an order of magnitude slower to train, due to the expensive computation of the preconditioner. For all solvers, we find that the test accuracy improves when simultaneously increasing $S$ and $\delta$, however the \modelAS{} ran out of memory for the last case, and the \modelLSQR{} performed similarly to our model, likely because the condition number of \mat{M} is relatively low in all cases.
}

\begin{figure}[!t]
\centering
\includegraphics[width=\textwidth]{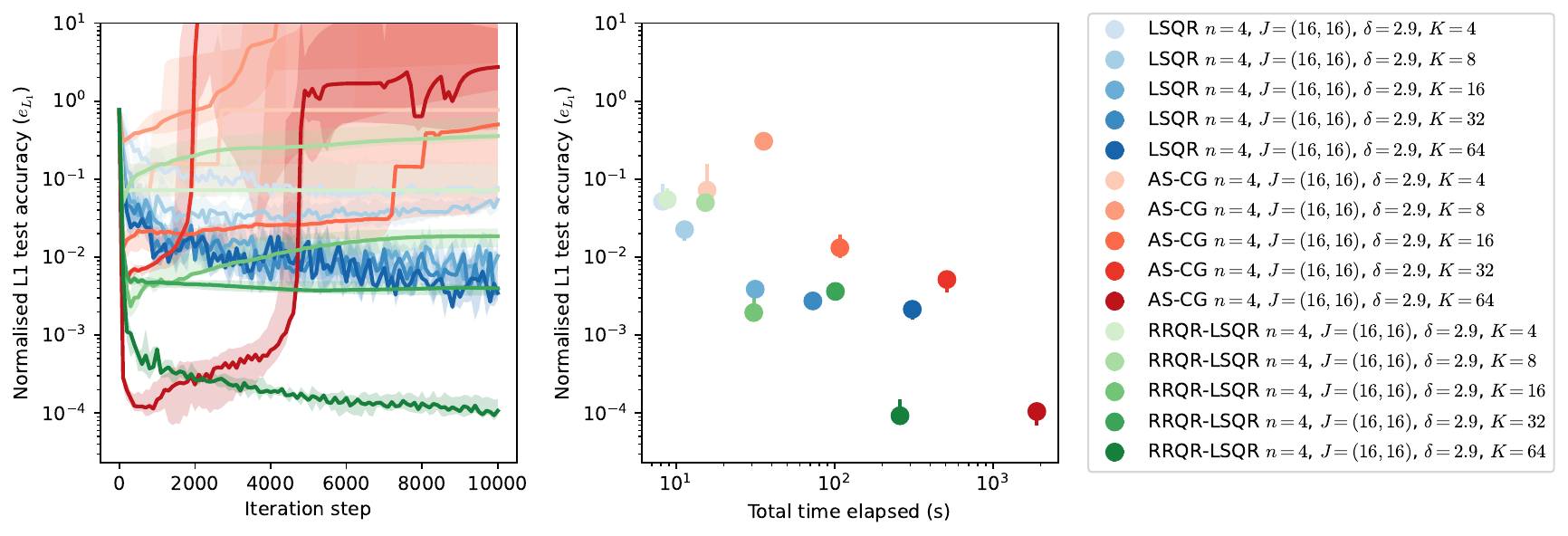}
\caption{Strong scaling test for the multi-scale Laplacian in 2D, fixing the number of solution components, $n$, and varying the number of basis functions, $K$.
\label{fig:laplace-strong-p-scaling}}
\end{figure}

\begin{figure}[!t]
\centering
\includegraphics[width=\textwidth]{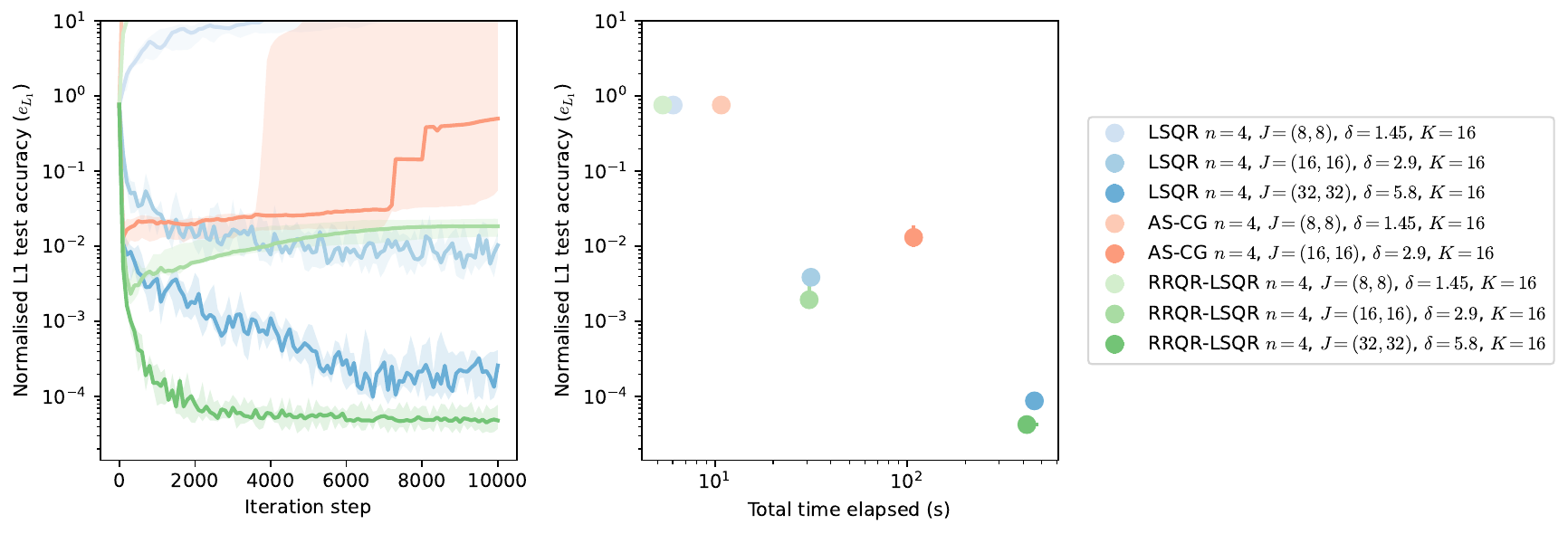}
\caption{Strong scaling test for the multi-scale Laplacian in 2D, fixing the number of solution components, $n$, and simultaneously varying the number of subdomains, $J$, and the overlap between subdomains, $\delta$.
\label{fig:laplace-strong-wm-scaling}}
\end{figure}

\paragraph{Weak scaling test}{We perform a weak scaling test in a similar fashion to the harmonic oscillator test, increasing the problem complexity by varying $n$ between $1, 2, 3, 4, 5,\text{ and }6$, whilst simultaneously increasing $S$ between $2, 4, 8, 16, 32\text{ and }64$. Plots of the convergence curves and test accuracy versus total training time of the \model{}, \modelLSQR{}, and \modelAS{} for each case are shown in \Cref{fig:laplace-weak-scaling}, and the exact solution and \model{} solution are plotted for different $n$ values in \Cref{fig:laplace-weak-solution}. We find that all solvers are able to solve all cases to a reasonable accuracy, but undesirably their accuracy reduces as $n$ increases. The final case has $J=64\times64=4,096$ subdomains; we postulate that the worsening error is because adding subdomains without increasing their overlap makes it challenging to communicate between subdomains, potentially making the model harder to train, as studied by \cite{Dolean:MDD:2024}. Increasing the overlap reduces sparsity in $\mat{M}$, significantly increasing memory and training time, such that we were unable to increase the overlap significantly without running out of memory. In future work, a more promising direction may be to instead aid communication between subdomains by incorporating multiple levels of domain decompositions, as studied by \cite{Dolean:MDD:2024}.

\begin{figure}[!t]
\centering
\includegraphics[width=\textwidth]{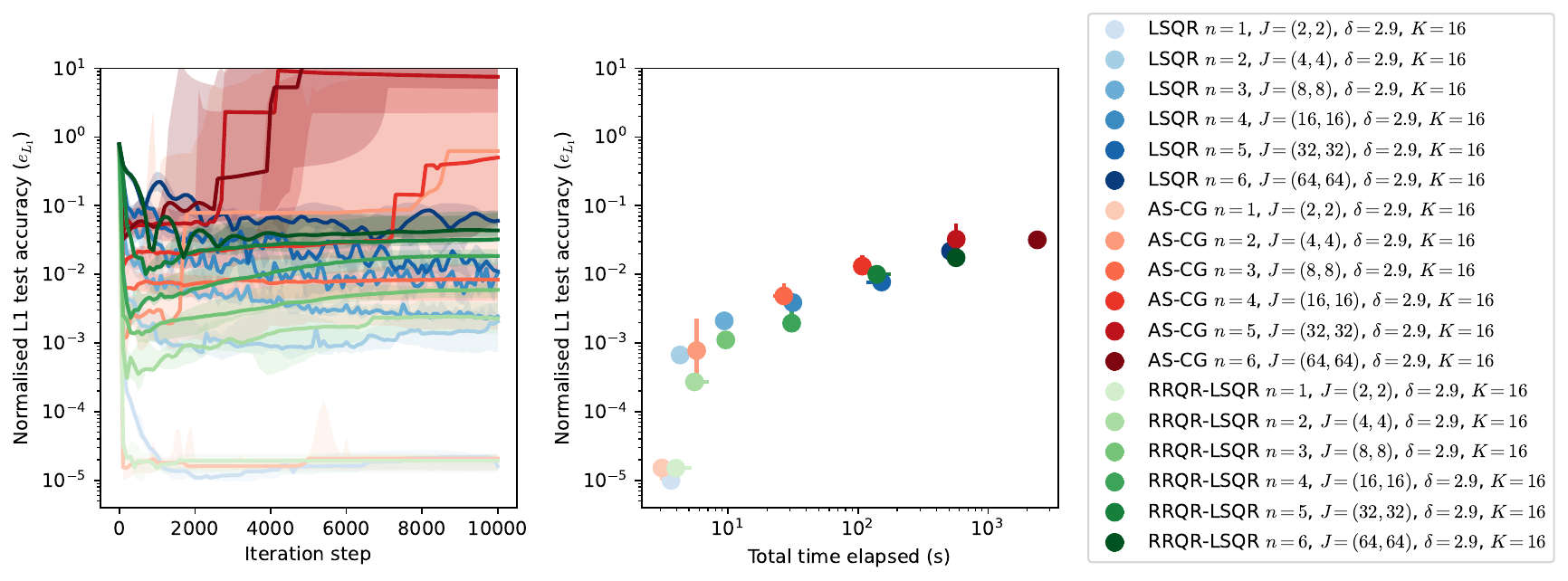}
\caption{Weak scaling test for the multi-scale Laplacian in 2D, increasing the number of solution components, $n$, whilst simultaneously increasing the number of subdomains, $J$.
\label{fig:laplace-weak-scaling}}
\end{figure}

\begin{figure}[!t]
\centering
\includegraphics[width=0.75\textwidth]{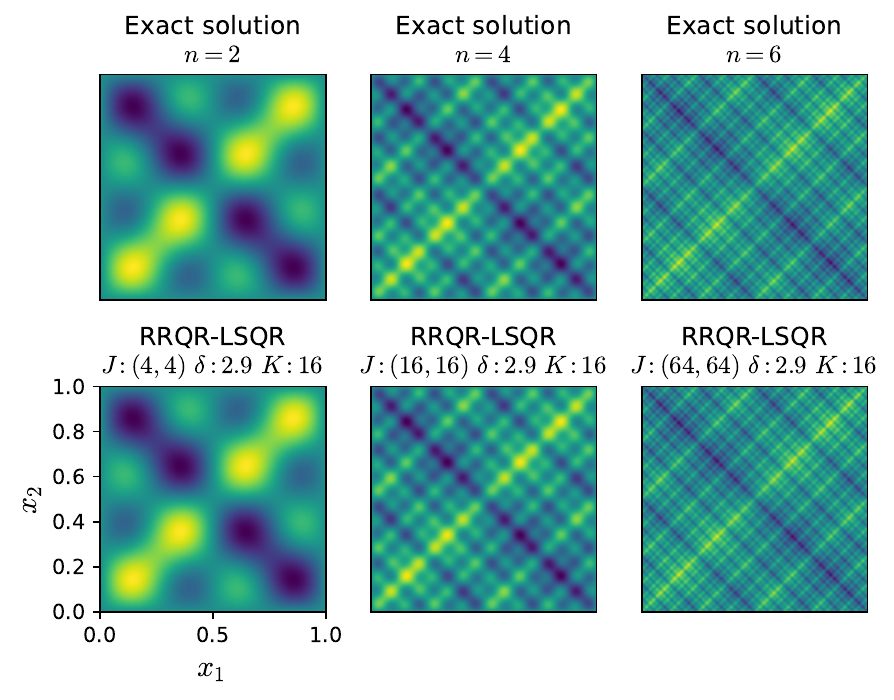}
\caption{Comparison of exact solution and \model{} solution for three of the cases in the weak scaling test for the multi-scale Laplacian in 2D.
\label{fig:laplace-weak-solution}}
\end{figure}

\subsubsection{Time-varying wave equation in (2+1)D}\label{sec:wave_equation_results}

\paragraph{Baseline model}{Finally, we study the (2+1)D wave equation problem. We train a baseline \model{}, studying the problem with $\lambda=0.2$, using $S=8$, $\delta=2.9$, $K=8$, a subdomain network depth of $h=1$, $\tanh$ activation function, $\sigma=\SI{1e-8}{}$, and $4\times S=32$ collocation points along each input dimension for testing. A comparison of the solution from FD simulation and our model solution is shown in \Cref{fig:wave-equation-weak-solution}. We find that the model is able to predict the solution with reasonable accuracy, with a normalized L1 test accuracy of $2.2\pm\SI{0.4e-02}{}$. Similarly to the multi-scale Laplacian problem, we find that no columns are dropped during $\sigma$-RRQR filtering, likely because the condition number of $\mat{M}$ is already relatively low at \SI{5.1e+6}{}. After preconditioning the condition number of $\widehat{\mat{M}} \mat{S}^{-1}$ further drops to \SI{2.4e+02}{}. The convergence of the LSQR solver is shown in \Cref{fig:wave-equation-convergence}. We find that the model converges within a few tens of steps, whilst the total time taken to carry out the full 10,000 iterations is $1.0\pm\SI{0.0e02}{s}$.
}

\begin{figure}[!t]
\centering
\includegraphics[width=\textwidth]{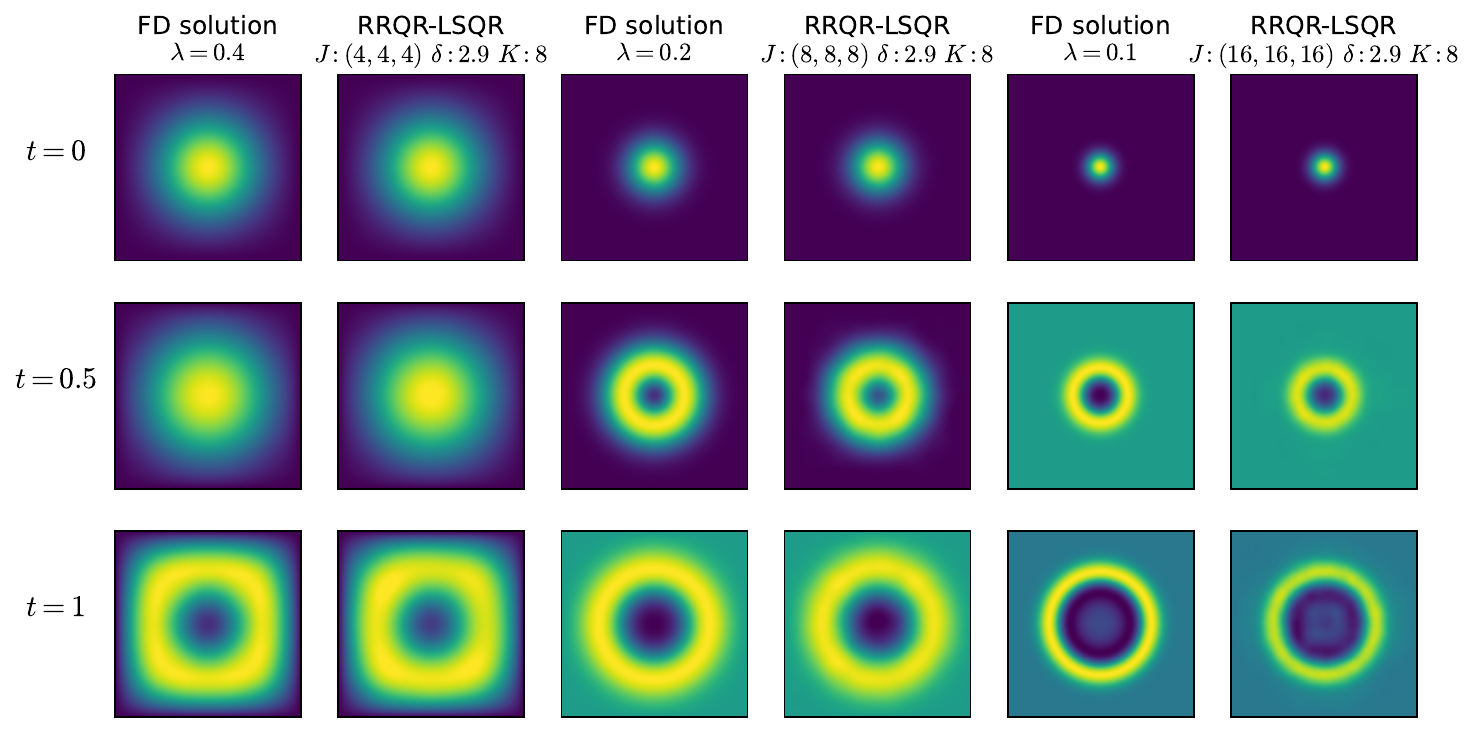}
\caption{Comparison of exact solution and \model{} solution for three of the cases in the weak scaling test for the time-varying wave equation in (2+1)D.
\label{fig:wave-equation-weak-solution}}
\end{figure}

\paragraph{Baseline comparison}{We compare our baseline \model{} above to our baseline models. A comparison of the relevant condition numbers, test accuracy and total training time is shown in \Cref{tab:wave-equation}. In this case, all baseline models achieve a similar level of accuracy, but the training time for the PINN and FBPINN models is two orders of magnitude larger than the linear solver models. The convergence curves of all linear solver models are compared in \Cref{fig:wave-equation-convergence}. Similar to the harmonic oscillator and multi-scale Laplacian problem, we find that the \modelsNoPre{} take the full 10,000 iteration steps to converge, whilst both our model and the \modelAS{} converge within a few tens of steps.
}

\begin{table}[h]
\setlength{\tabcolsep}{2pt}
\centering
\resizebox{\textwidth}{!}{
\begin{tabular}{ccccccccccc}
\toprule
Network & $(h,K)$ & $\kappa(M)$ & $\kappa(\widehat{M})$ & $\kappa(\widehat{M} S^{-1})$ & $\kappa(M^T M)$ & $\kappa(A^{-1}_{AS} A)$ & Optimiser & $\sigma$ ($\phi_{jk}$ drop \%) & $e_{L_1}$ & Time (s) \\
\midrule
FCN & (2, 128) & N/A & N/A & N/A & N/A & N/A & PINN-Adam & N/A (N/A) & 2.9$\pm$1.8e-02 & 1.1$\pm$0.0e+04 \\
FCN & (1, 8) & N/A & N/A & N/A & N/A & N/A & FBPINN-Adam & N/A (N/A) & 1.9$\pm$0.9e-02 & 1.0$\pm$0.0e+04 \\
ELM & (1, 8) & 5.1e+06 & N/A & N/A & N/A & N/A & LSQR & N/A (N/A) & 2.4$\pm$0.3e-02 & 1.2$\pm$0.1e+02 \\
ELM & (1, 8) & 5.1e+06 & N/A & N/A & 2.6e+13 & N/A & CG & N/A (N/A) & 2.4$\pm$0.3e-02 & 10.0$\pm$1.3e+01 \\
ELM & (1, 8) & 5.1e+06 & N/A & N/A & 2.6e+13 & 3.8e+08 & AS-CG & N/A (N/A) & 2.3$\pm$0.4e-02 & 6.2$\pm$0.1e+02 \\
ELM & (1, 8) & 5.1e+06 & 5.1e+06 & 2.4e+02 & N/A & N/A & RRQR-LSQR & 1e-08 (0\%) & 2.2$\pm$0.4e-02 & 1.0$\pm$0.0e+02 \\
ELM & (1, 8) & 5.1e+06 & 5.1e+06 & 2.4e+02 & N/A & N/A & RRQR-LSQR & 1e-10 (0\%) & 2.2$\pm$0.4e-02 & 8.5$\pm$0.1e+01 \\
ELM & (1, 8) & 5.1e+06 & 5.1e+06 & 2.4e+02 & N/A & N/A & RRQR-LSQR & 1e-06 (0\%) & 2.2$\pm$0.4e-02 & 1.3$\pm$0.3e+02 \\
ELM & (1, 8) & 5.1e+06 & 8.7e+05 & 2.4e+02 & N/A & N/A & RRQR-LSQR & 1e-04 (0\%) & 2.2$\pm$0.4e-02 & 1.3$\pm$0.3e+02 \\
ELM & (1, 8) & 5.1e+06 & 1.1e+04 & 8.2e+01 & N/A & N/A & RRQR-LSQR & 1e-02 (38\%) & 4.7$\pm$1.1e-02 & 6.8$\pm$0.1e+01 \\
ELM-$\sigma$ & (1, 8) & 9.9e+07 & 9.9e+07 & 2.9e+02 & N/A & N/A & RRQR-LSQR & 1e-08 (0\%) & 2.3$\pm$0.5e-02 & 8.4$\pm$0.2e+01 \\
ELM-$\sin$ & (1, 8) & 8.9e+06 & 8.9e+06 & 2.9e+02 & N/A & N/A & RRQR-LSQR & 1e-08 (0\%) & 2.3$\pm$0.5e-02 & 9.3$\pm$0.1e+01 \\
ELM & (3, 8) & 9.5e+06 & 9.5e+06 & 2.5e+02 & N/A & N/A & RRQR-LSQR & 1e-08 (0\%) & 2.2$\pm$0.2e-02 & 1.6$\pm$0.1e+02 \\
ELM & (5, 8) & 1.6e+08 & 1.6e+08 & 2.3e+02 & N/A & N/A & RRQR-LSQR & 1e-08 (0\%) & 2.4$\pm$0.8e-02 & 1.4$\pm$0.0e+02 \\
\bottomrule
\end{tabular}

}
\caption{Comparison of test error, total training time, and conditioning numbers for different models solving the time-varying wave equation in (2+1)D.}
\label{tab:wave-equation}
\end{table}

\begin{figure}[!t]
\centering
\includegraphics[width=0.75\textwidth]{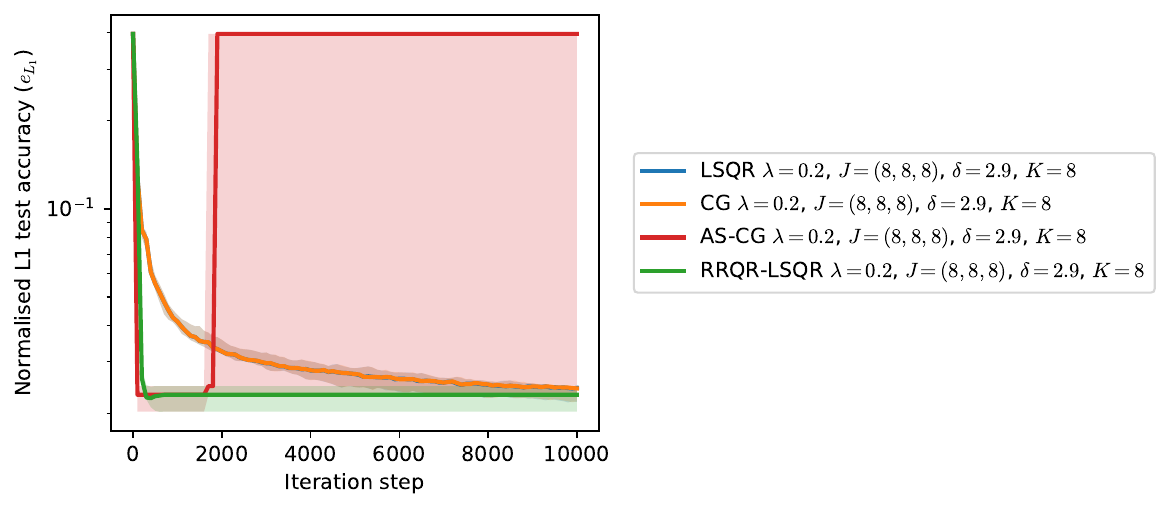}
\caption{Convergence curves of the \model{} and our baseline models when solving the time-varying wave equation in (2+1)D.
\label{fig:wave-equation-convergence}}
\end{figure}

\paragraph{Ablation tests}{
We perform the same suite of ablation tests on our model as above. The test accuracy and total training time of these models is reported in \Cref{tab:wave-equation}. The same conclusions as the multi-scale Laplacian problem can be drawn.
}

\paragraph{Strong scaling tests}{
We perform a strong scaling test, fixing the problem complexity by keeping $\lambda=0.2$, and increasing the model capacity first by varying $K$ between $4,8,16,\text{ and }32$, and second by fixing $K=8$, but simultaneously varying $S$ between $4, 8, \text{ and }16$ and $\delta$ between $1.45, 2.9, \text{ and }5.8$. Plots of the convergence curves and test accuracy versus total training time of the \model{}, \modelLSQR{}, and \modelAS{} for each scaling test are shown in \Cref{fig:wave-equation-strong-p-scaling} and \Cref{fig:wave-equation-strong-wm-scaling}. For all solvers, we find that the test accuracy improves when increasing $K$, showing good scaling behavior. However, the \modelAS{} is approximately an order of magnitude slower to train, and runs out of memory for the last case. Furthermore, the convergence rate of our model significantly worsens as $K$ increases. When using a regular overlapping rectangular domain decomposition, the maximum number of overlapping models increases exponentially with the number of input dimensions $d$, so we postulate that our block-wise preconditioning may not be as effective with high $d$ because it neglects the overlapping contributions between subdomains in $\mat{M}$. Improving the performance of our preconditioner in higher dimensions will be the subject of future work. For all solvers, similar to the harmonic oscillator and multi-scale Laplacian problem, we find that the test accuracy improves when simultaneously increasing $S$ and $\delta$.
}

\begin{figure}[!t]
\centering
\includegraphics[width=\textwidth]{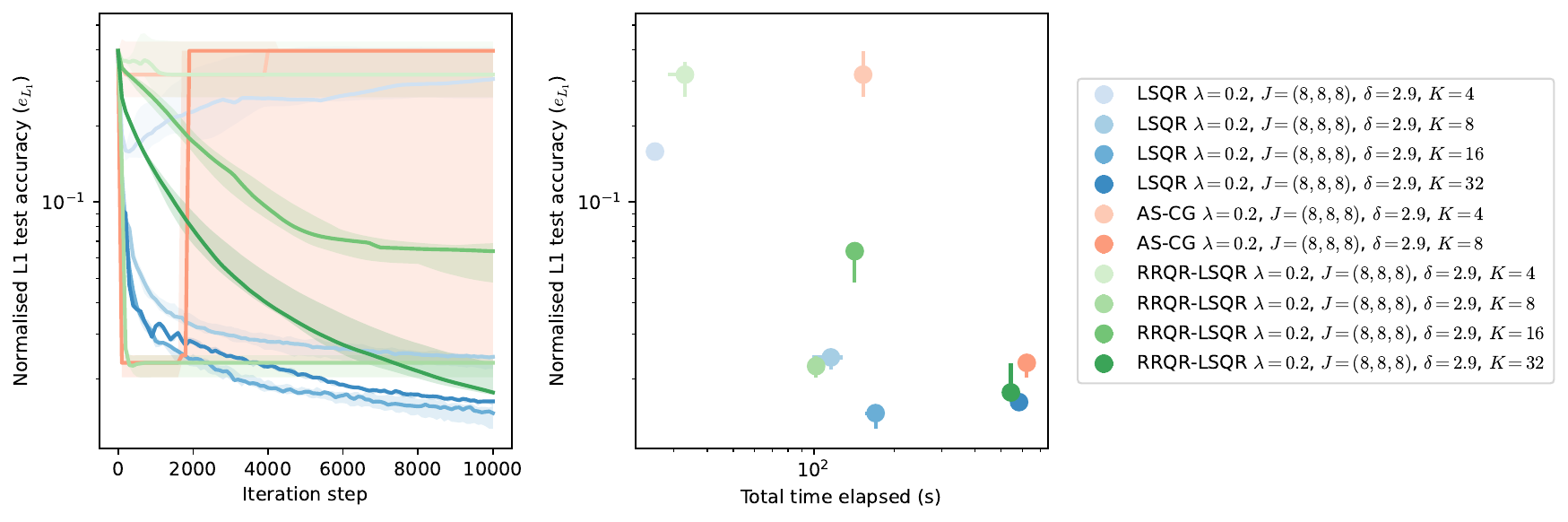}
\caption{Strong scaling test for the time-varying wave equation in (2+1)D, fixing the solution wavelength, $\lambda$, and varying the number of basis functions, $K$.
\label{fig:wave-equation-strong-p-scaling}}
\end{figure}

\begin{figure}[!t]
\centering
\includegraphics[width=\textwidth]{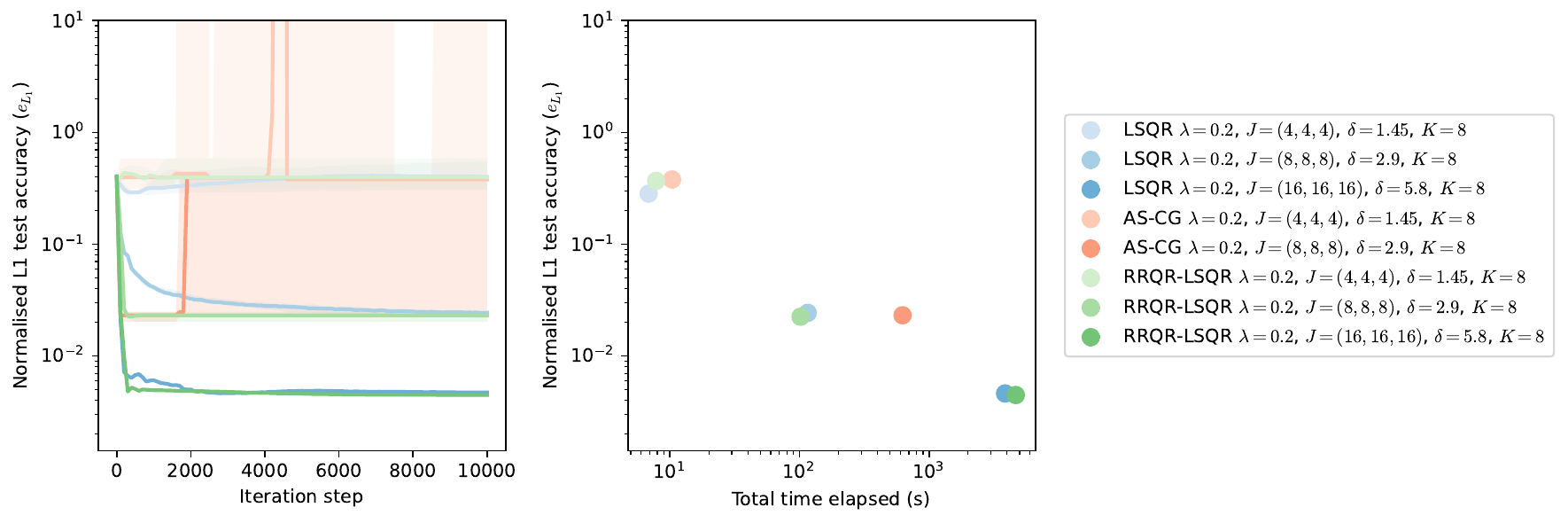}
\caption{Strong scaling test for the time-varying wave equation in (2+1)D, fixing the solution wavelength, $\lambda$, and simultaneously varying the number of subdomains, $J$, and the overlap between subdomains, $\delta$.
\label{fig:wave-equation-strong-wm-scaling}}
\end{figure}

\paragraph{Weak scaling test}{Finally, we perform a weak scaling test, increasing the problem complexity by varying $\lambda$ between $0.4, 0.2\text{ and }0.1$, whilst simultaneously increasing $S$ between $4, 8,\text{ and }16$. Plots of the convergence curves and test accuracy versus total training time of the \model{}, \modelLSQR{}, and \modelAS{} for each case are shown in \Cref{fig:wave-equation-weak-scaling}, and the exact solution and \model{} solution are plotted for different $\lambda$ values in \Cref{fig:wave-equation-weak-solution}. Similarly to the multi-scale Laplacian problem, we find that all solvers are able to solve all cases to a reasonable accuracy, but with their accuracy reducing as $\lambda$ decreases. As suggested above, improving communication between subdomains by adding multiple levels of domain decompositions may improve performance further.

\begin{figure}[!t]
\centering
\includegraphics[width=\textwidth]{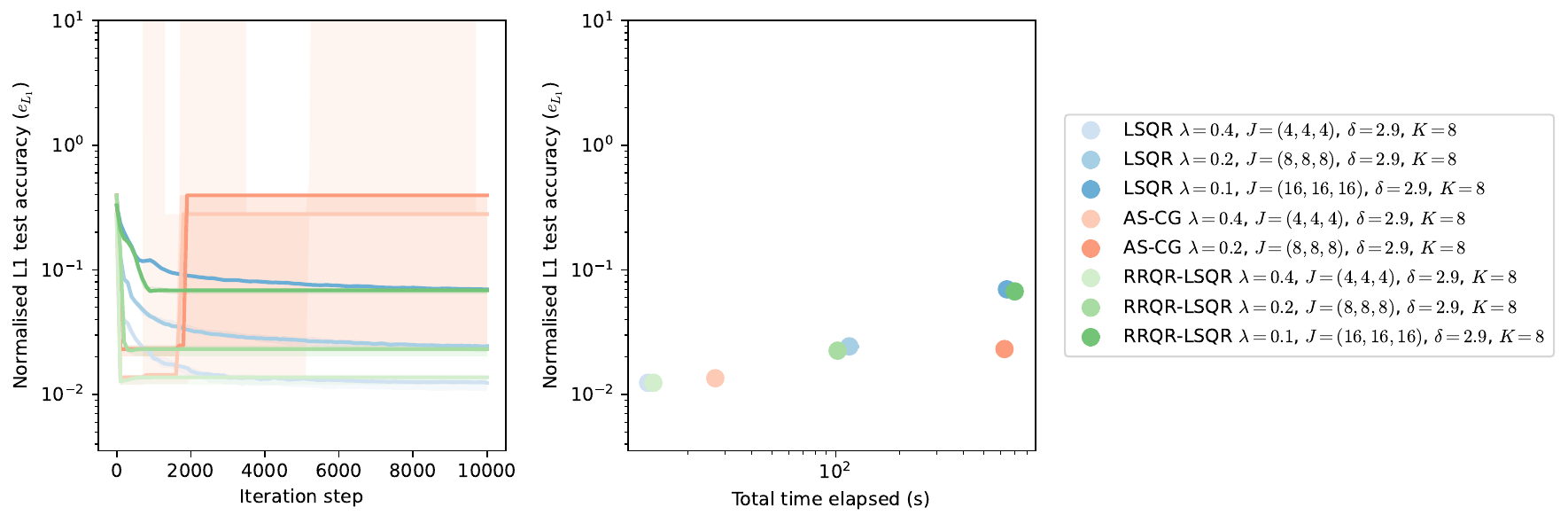}
\caption{Weak scaling test for the time-varying wave equation in (2+1)D, increasing the solution wavelength, $\lambda$, whilst simultaneously increasing the number of subdomains, $J$.
\label{fig:wave-equation-weak-scaling}}
\end{figure}

\subsection{General observations and future work}
\label{sec:discussion}

Across all experiments, we find that our $\sigma$-RRQR filtering and preconditioning strategy significantly improves both the test accuracy and convergence rate of ELM-FBPINNs compared to our baseline models. Our filtering and preconditioning always improves the conditioning of the ELM-FBPINN least-squares system, enabling faster convergence rates, with filtering having the added benefit of reducing the size of $\mat{M}$ and therefore computational time of the solver. Compared to PINNs and FBPINNs, our solver is multiple orders of magnitude faster, because it does not need to rely on nonlinear gradient descent. Compared to \modelsNoPre{}, our solver generally has a much faster rate of convergence because of preconditioning. Compared to \modelAS{}, our solver requires significantly less time and memory to compute the preconditioner, and exhibits more stable convergence curves, possibly because we drop highly numerically dependent columns and directly solve the least-squares system instead of its normal equations. 

We identified multiple directions for future work and outstanding limitations of our method, including improving performance for larger numbers of subdomains, perhaps by extending our framework to multilevel domain decompositions, as proposed by \cite{Dolean:MDD:2024}, and investigating improvements for higher numbers of input dimensions, perhaps by incorporating overlapping information into our preconditioner. It is interesting to note that whilst our preconditioning always improves the conditioning of the system, it does not always improve the final test error achieved by the solver. Further work could be carried out to understand this behavior more. Future work could also include further accelerating our solvers, which are currently implemented in SciPy, by implementing them on the GPU, and building further theoretical understanding on the convergence rate of ELM-FBPINNs.


\appendix

\section{Probabilistic refinements via random matrix theory}
\label{appendix:RMT}

While the main analysis in Section~\ref{sec:contribution} is deterministic, it is often  instructive to understand the \emph{typical-case} behaviour of the preconditioned system 
when the features are randomized. In practice, the subdomain matrices obtained after local  RRQR filtering tend to have nearly orthogonal columns and exhibit statistical independence 
across subdomains, a consequence of the random feature construction.  To formalize this intuition, we model the filtered subdomain blocks as random  semi-orthogonal matrices. 

\subsection{Random orthogonal matrices}
Specifically, we model the subdomain-restricted matrices produced by RRQR as random semi-orthogonal matrices, reflecting their approximate local orthogonality and statistical independence due to randomized features. This abstraction allows us to analytically study how the conditioning of the global system changes as a function of overlap size. In this section, we collect relevant results on random orthogonal and semi-orthogonal matrices that will be used later in this section.
We will denote random matrices with calligraphic letters (i.e. $\mathcal{Q}$ instead of $Q$) to distinguish them from ordinary matrices.

We start by recalling a definition of a natural uniform distribution on the group $\mathcal{O}_n(\R)$ \cite{Stewart_1980}:

\begin{definition}[Random orthogonal matrix] 
\label{def:DefOrthMat}
Let $\mathcal{O}_n(\mathbb{R})$ denote the group of $n \times n$ orthogonal matrices over $\mathbb{R}$. There exists a unique probability measure $\mu$ on $\mathcal{O}_n(\mathbb{R})$ such that if $\mathcal{Q} \sim \mu$ (i.e. is distributed according to $\mu$), then for all $U, V \in \mathcal{O}_n(\mathbb{R})$, the matrix $U \mathcal{Q} V$ has the same distribution as $\mathcal{Q}$. A matrix drawn from this distribution is called a \emph{random orthogonal matrix} distributed according to the Haar measure on $\mathcal{O}_n(\mathbb{R})$.
\end{definition}

The following useful and constructive characterization of such matrices is also given in \cite{Stewart_1980}. 

\begin{proposition}[Characterisation of random orthogonal matrices]   Let $\mathcal{X}$ be an $n \times n$ matrix with entries independently distributed according to a standard normal distribution.
Let $\mathcal{X} = \mathcal{Q}\mathcal{R}$ be the $QR$-factorization, where all diagonal elements of $\mathcal{R}$ are normalized to be positive. Then $\mathcal{Q}$ is a random orthogonal matrix according to the Haar measure on $\mathcal{O}_n(\mathbb{R})$. 
\end{proposition}

From this it follows that for each column $\mathcal{Q}^{(i)}$ separately, one has
\begin{equation} \label{eq:GaussNorm}
\mathcal{Q}_j^{(i)} = \frac{X_j}{\sqrt{\sum_{k=1}^n X_k^2}},
\end{equation}

with the $X_j$ independent standard normal random variables.
This is certainly true for the first column, as the first step of Gramm-Schmidt simply normalizes the first column of $\mathcal{X}$. By the invariance of $\mathcal{Q}$ under column permutations (which follows from definition \ref{def:DefOrthMat}), it then follows for the other columns as well.

\begin{definition}[Random semi-orthogonal matrix]
Let $\mathcal{Q}$ be a random orthogonal matrix of size $m \times m$, and let $\mathcal{A}$ be the $m \times k$ matrix obtained by selecting any $k \leq m$ columns of $\mathcal{Q}$. Then $\mathcal{A}$ is called a \emph{random semi-orthogonal matrix}.
\end{definition}
By \Cref{def:DefOrthMat}, the distribution of $\mathcal{A}$ does not depend on the particular choice of columns of $\mathcal{Q}$. These matrices are also called $k$-frames in the random matrix literature \cite{Tropp_2011}. We call them random semi-orthogonal here to stress the fact that $\mathcal{A}^T \mathcal{A} = I$. 

The following result, i.e. that the product of two independent Haar-distributed orthogonal matrices is itself a Haar-distributed is not usually presented as standalone theorem in the relevant literature \cite{DiaconisShahshahani1987, Meckes2019} but rather as a consequence of the invariance properties of the Haar measure. It is considered standard and foundational in the theory of probability on groups. In order to make the presentation self consistent and accessible in the current context we state this here and give a short proof for random orthogonal matrices.

\begin{proposition}[Invariance under multiplication] \label{prop:OrthInvariant}
Let $\mathcal{P}, \mathcal{Q} \in \mathcal{O}_n(\mathbb{R})$ be independent random orthogonal matrices. Then their product $\mathcal{Z} = \mathcal{P} \mathcal{Q}$ is also a random orthogonal matrix. In particular, $\mathcal{Z}$ has the same distribution as $\mathcal{P}$ and $\mathcal{Q}$.
\end{proposition}

\begin{proof}
The key property of the Haar measure is its invariance under both left and right multiplication by fixed orthogonal matrices. Concretely, if $\mathcal{Q}$ is Haar-distributed and $P \in \mathcal{O}_n(\mathbb{R})$ is fixed, then the product $P\mathcal{Q}$ has the same distribution as $\mathcal{Q}$. This invariance implies that multiplying a Haar-distributed matrix by a fixed orthogonal matrix does not alter its distribution.

Now, let $\mathcal{P}$ and $\mathcal{Q}$ be independent Haar-distributed orthogonal matrices, and define $\mathcal{Z} = \mathcal{P}\mathcal{Q}$. We want to show that $\mathcal{Z}$ is also Haar-distributed.

Let $H \subset \mathcal{O}_n(\mathbb{R})$ be any measurable subset. Conditioning on a fixed realization $\mathcal{P} = P$, we have:
\[
\mathbb{P}(\mathcal{Z} \in H \mid \mathcal{P} = P) 
= \mathbb{P}(P\mathcal{Q} \in H)
= \mathbb{P}(\mathcal{Q} \in P^{-1}H),
\]
where \( P^{-1}H = \{Q \in \mathcal{O}_n \mid PQ \in H\} \). Since $\mathcal{Q}$ is Haar-distributed and $P$ is fixed, we have
\[
\mathbb{P}(\mathcal{Q} \in P^{-1}H) = \mathbb{P}(\mathcal{Q} \in P^TH) = \mathbb{P}(\mathcal{Q} \in H),
\]
by left invariance of the Haar measure.

Integrating over all possible values of $\mathcal{P}$, we obtain:
  \begin{align*}
    \mathbb{P}(\mathcal{Z} \in H) &= \int_{\mathcal{O}_n} \mathbb{P}(\mathcal{Z} \in H \mid \mathcal{P}=P )\ d \mu(P) = \int_{\mathcal{O}_n} \mathbb{P}(\mathcal{Q} \in H)\ d\mu(P) \\
      &= \mathbb{P}(\mathcal{Q} \in H) \int_{\mathcal{O}_n} d \mu(P) = \mathbb{P}(\mathcal{Q} \in H). 
  \end{align*}
since $\mu$ is a probability measure. Therefore, $\mathcal{Z}$ is Haar-distributed.
\end{proof}

\begin{remark}[Numerical intuition]
This result formalizes an intuitive property: if we multiply two orthogonal matrices, each randomly sampled in a systematic way (e.g., using the QR decomposition of Gaussian matrices), the result is still randomly distributed in the same way in the space of orthogonal matrices. This closure under multiplication is a feature of compact groups like $\mathcal{O}_n(\mathbb{R})$, and ensures that the "orthogonal randomness" is preserved under composition, an important property in randomized algorithms and preconditioner analysis. We say that the statistical properties of such matrices are preserved under orthogonal transformations and compositions and their use in iterative solvers or algorithms does not introduce additional bias or artifacts.
\end{remark}

For our application, we are interested in the operator norm of the product of blocks of orthogonal matrices. A precise investigation of this quantity is beyond the scope of this paper. Instead, we give a bound on the expected Frobenius norm, which is easier to handle. 

\begin{proposition}[Expected Frobenius norm of product] \label{proposition:frobenius_bound}
Let $\mathcal{P}, \mathcal{Q} \in \mathcal{O}_n(\mathbb{R})$ be independent random orthogonal matrices drawn from the Haar measure. Let $\mathcal{P}_1, \mathcal{Q}_1$ be $\ell \times K$ blocks of $\mathcal{P}$ and $\mathcal{Q}$ respectively. Then we have 
\[ \mathbb{E}\| \mathcal{Q}_1^T \mathcal{P}_1 \|_{\mathrm{F}} \leq \frac { K\sqrt{\ell}}{n}. \]

\end{proposition}
\begin{proof}
Let $\mathcal{A} = \mathcal{Q}_1^T \mathcal{P}_1$, with entries $a_{ij}$ for $i,j \ \in \{1, \cdots, K\}$. By \ref{eq:GaussNorm}, the columns of $\mathcal{Q}_1$ have entries of the form 
\[ \mathcal{Q}_k^{(m)} = \frac{X_k}{\sqrt{\sum_{k=1}^n X_k^2}},
\]
with the $X_k$ independent standard normal random variables.
The columns of $\mathcal{P}_1$ have the same form. Hence an entry $a_{ij}$ of $\mathcal{A}$ is of the form
\[ a_{ij} = \sum_{k = 1}^{\ell} \frac{X_k Y_k}{\sqrt{\sum_{k=1}^n X_k^2} \sqrt{\sum_{k=1}^n Y_k^2}},\]
with the $X_k, Y_k$ independent standard normal. Taking the square and expectation of this expression yields
\[ \mathbb{E}\ a_{ij}^2 = \sum_{k = 1}^{\ell} \left (\mathbb{E}\frac{X_k^2}{\sum_{k=1}^n X_k^2} \mathbb{E} \frac{Y_k^2}{\sum_{k=1}^n Y_k^2} \right ) = \sum_{k = 1}^{\ell} \frac{1}{n^2} = \frac{\ell}{n^2}.\]
Note that we can eliminate the cross terms in the above square, since they have expectation 0. It follows that 
\[ \mathbb{E} \|\mathcal{A} \|_{\mathrm{F}} = \mathbb{E} \sqrt{ \sum_{i,j \leq K} a_{ij}^2} \leq \sqrt {\mathbb{E}  \sum_{i,j \leq K} a_{ij}^2} = \sqrt{ \sum_{i,j \leq K} \mathbb{E}\ a_{ij}^2 } = \frac{ K \sqrt{\ell}}{n},\]
where we have used Jensen's inequality and the fact that taking the square root is a concave function.
\end{proof}

\begin{remark} It is easily calculated that 
\[\mathbb{E}\| \mathcal{P}_1 \|_{\mathrm{F}} = \mathbb{E}\| \mathcal{Q}_1 \|_{\mathrm{F}} = \sqrt{\frac{K\ell}{n}}.\] 

Applying the submultiplicativity of the Frobenius norm and using the independence of $\mathcal{P}_1$ and $\mathcal{Q}_1$ would give 
\[ \mathbb{E} \|\mathcal{Q}_1^T \mathcal{P}_1 \|_{\mathrm{F}} \leq \frac{K \ell}{n}.  \]
So the above derived bound is slightly tighter than this. 
\end{remark}
Finally we cite the following bound from \cite{Tropp_2011}, which uses far more sophisticated methods to estimate the spectral norm of a submatrix directly.

\begin{proposition}[Expected spectral norm of a principal block {\cite{Tropp_2011}}] \label{tropp_bound}
Let $\mathcal{P}_{1} \in \mathbb{R}^{\ell \times K}$ be the top-left block of a Haar-distributed orthogonal matrix $\mathcal{P} \in \mathcal{O}_n(\mathbb{R})$. Then
\[
\mathbb{E}\left\| \mathcal{P}_{1} \right\| \leq \left(1 + \frac{K}{2n} \right) \frac{\sqrt{K} + \sqrt{\ell}}{\sqrt{n}}.
\]
\end{proposition}
\begin{proof}
Although rigorous proof of this result is more involved we give here the main ideas. To estimate the expected norm of the principal block $\mathcal{P}_{1}$ of a Haar-distributed orthogonal matrix, we leverage the rotational invariance of the Haar measure, which implies that any fixed $\ell \times K$ block has the same distribution. The behavior of $\mathcal{P}_{1}$ is then compared to that of a normalized Gaussian matrix with i.i.d.\ $\mathcal{N}(0,1)$ entries, whose expected operator norm is known to be bounded by $\sqrt{K} + \sqrt{\ell}$. This yields a baseline estimate of $(\sqrt{K} + \sqrt{\ell})/\sqrt{n}$. Finally, Tropp \cite{Tropp_2011} refines this comparison by accounting for the difference between the true Haar marginal and the Gaussian approximation, introducing a correction factor $\left(1 + \frac{K}{2n}\right)$ that captures the residual dependence structure in $\mathcal{P}_{1}$.
\end{proof}
\begin{corollary} \label{corollary:spectral_bound} Let $\mathcal{P}_1$ and $\mathcal{Q}_1$ be as in proposition \ref{proposition:frobenius_bound}. Let $\varepsilon_{\ell,K,n}$ denote the quantity
\[\varepsilon_{\ell,K,n} =\left(1 + \frac{K}{2n}\right)
\frac{\sqrt{K} + \sqrt{\ell}}{\sqrt{n}}. \]
Then we have
\[ 
\mathbb{E}\|\mathcal{Q}_1^T \mathcal{P}_1\|_2 \leq \varepsilon_{\ell,K,n}^2.
\]
\end{corollary}
\begin{proof}
This follows directly from proposition \ref{tropp_bound} and the submultiplicativity of the spectral norm.
\end{proof}
\begin{remark}
Applying the methods of \cite{Tropp_2011} directly to $\mathcal{Q}_1^T \mathcal{P}_1$ would likely lead to a tighter estimate than using the submultiplicativity of the spectral norm. However, this is beyond the scope of this paper.
\end{remark}


\subsection{Estimating block interactions via random matrix theory}
To carry out a probabilistic analysis, we focus on the interaction matrices \(\mat{A}_j \;=\; \mat{Q}_{j+1,j}^\top \mat{Q}_{j,j } 
\quad\in \mathbb{R}^{K\times K}
\) arising from the overlapping blocks \(\mat{Q}_{j+1,j}, \mat{Q}_{j,j} \ \in\ \mathbb{R}^{\ell \times K} \).  When the features within each subdomain are generated by a randomized construction, 
it is natural to model each block as a realisation of a random matrix \(\mathcal{Q}\) representing the \emph{top} \(\ell\) rows of the first \(K\) columns of an \(n\times n\) Haar-distributed orthogonal matrix \(\mathcal{P} \in \mathcal{O}_n\). As in the previous paragraph, let $\varepsilon_{\ell,K,n}$ denote 
\[\varepsilon_{\ell,K,n} =\left(1 + \frac{K}{2n}\right)
\frac{\sqrt{K} + \sqrt{\ell}}{\sqrt{n}}. \]
Applying \ref{corollary:spectral_bound} to $\mat{A}_j$ yields
\begin{equation} \label{eq:Aj-bound-with-ell}
\mathbb{E} \|\mat{A}_j\|_2 
\ \leq
\varepsilon_{\ell,K,n}^2.
\end{equation}

\paragraph{From $\alpha$ to $\kappa(\mat{Q})$}
The deterministic bound in Theorem~\ref{th:cond-bound} holds with
\(\alpha = \max_j \| \mat{A}_j \|_2\).  
Approximating $\alpha$ by its typical RMT value $\varepsilon_{\ell,K,n}$,  
and recalling that all $\mat{A}_j$ are identically distributed and independent,  
we arrive at the probabilistic estimate
\begin{equation}
\label{eq:cond-prob}
\mathbb{E}[\kappa(\mat{Q})]
\;\lesssim\;
\left( \frac{1 + 2\varepsilon_{\ell,K,n}^2}{1 - 2\varepsilon_{\ell,K,n}^2} \right)^{1/2},
\quad
\text{valid when}\quad 2\varepsilon_{\ell,K,n} < 1.
\end{equation}

\begin{remark}
The quantity $\varepsilon_{\ell,K,n}$ in~\Cref{eq:Aj-bound-with-ell} encodes the 
\emph{relative feature complexity} $K/n$ and the \emph{relative overlap size} $\ell/n$.
Small $K$ and small $\ell$ produce small $\varepsilon_{\ell,K,n}$ and hence 
a well-conditioned $\mat{Q}$.
Conversely, if either $K/n$ or $\sqrt{\ell/n}$ is large, then $\varepsilon_{\ell,K,n}^2$
approaches $1/2$ and the bound~\eqref{eq:cond-prob} in no longer effective.
From~\eqref{eq:cond-prob} we can distinguish two regimes:
\begin{itemize}
    \item \textbf{Small-overlap regime:}  
    If $\ell \ll n$ and $K \ll n$, then $\varepsilon_{\ell,K,n} \to 0$ and  
    \[
    \mathbb{E}[\kappa(\mat{Q})] \ \approx\ 1 +2\varepsilon_{\ell,K,n}^2 
    + \mathcal{O}(\varepsilon_{\ell,K,n}^4),
    \]
    indicating very mild conditioning loss.
    \item \textbf{Large-overlap regime:}  
    As $\ell/n$ increases, $\varepsilon_{\ell,K,n}$ grows and 
    the denominator $1 - 2\varepsilon_{\ell,K,n}^2$ in~\eqref{eq:cond-prob} shrinks.  
    When $\varepsilon_{\ell,K,n}^2 \to 1/2$, the bound formally blows up and beyond that limit i.e. for $\ell$ a significant fraction of $n$, the estimate is no longer applicable. 
\end{itemize}
\end{remark}

\paragraph{Numerical validation of bound}{
\Cref{eq:Aj-bound-with-ell} and \Cref{eq:cond-prob} provide a weak upper bound on $\kappa(\mat{Q})$. Specifically, the bound increases with $K$, the number of basis functions in the model, and $\ell$, the number of overlapping rows between each subdomain. We compare this result to how $\kappa(\mat{Q})$ changes numerically when increasing $K$ and $\delta$ (which analogously to $\ell$ controls the overlap between subdomains, as defined in \Cref{eq:dd}) when solving the harmonic oscillator in 1D} studied in \Cref{sec:harmonic_oscilator_results}. We define a \model{} using the same baseline hyperparameters described in \Cref{sec:harmonic_oscillator_baseline}, except that we fix $\sigma=0$ and the number of interior collocation points to be $2000$. \Cref{fig:harmonic-oscillator-appendix-scaling} shows how $\kappa(\mat{M})$ and $\kappa(\mat{Q})$ vary when changing $K$ and $\delta$. We see the same behavior expected by the bound; both condition numbers increase significantly with increasing $K$ and $\delta$. 
}

\begin{figure}[!t]
\centering
\includegraphics[width=\textwidth]{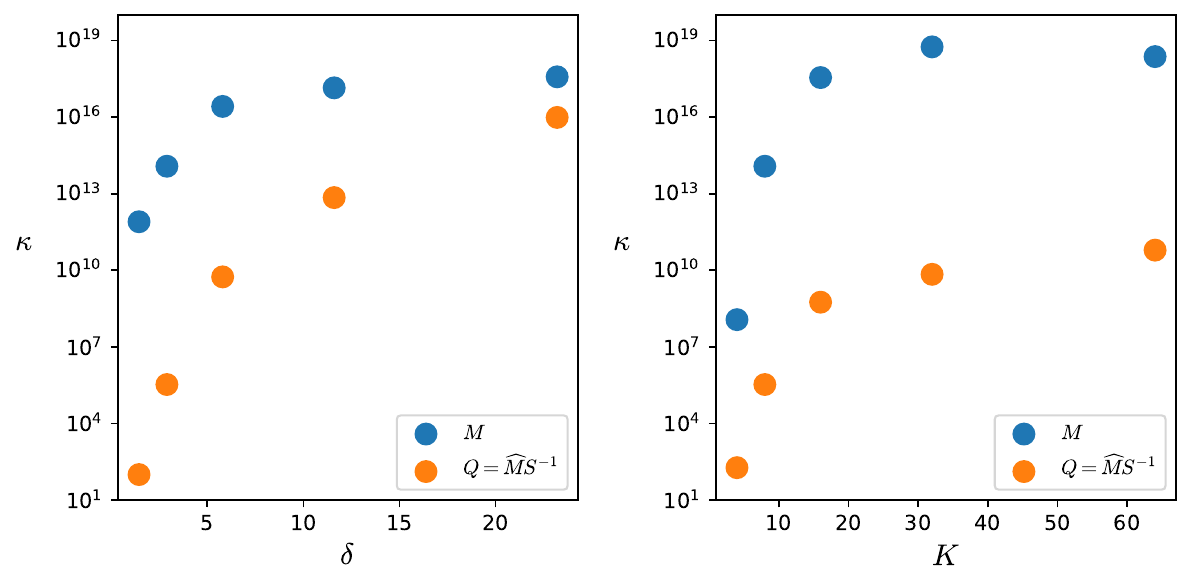}
\caption{Plots showing how the condition number of $\mat{M}$ and $\mat{Q}$ vary with $K$ and $\delta$ when solving the harmonic oscillator in 1D with our baseline \model{} described in \Cref{sec:harmonic_oscillator_baseline}.
\label{fig:harmonic-oscillator-appendix-scaling}}
\end{figure}

\section{Finite difference modeling for time-varying wave equation in (2+1)D}
\label{sec:appendix_fd}

We use finite difference modeling as the ground truth solution when studying the (2+1)D wave equation problem in \Cref{sec:wave_equation_results}. We use the \verb|SEISMIC_CPML| library \cite{Komatitsch2007} which performs staggered-grid second order finite difference modeling of the time-dependent 2D acoustic wave equation with Dirichlet boundary conditions. For ease of use, we re-implemented the original Fortran code in Python. The simulation is initialized by sampling the initial wavefield and initial wavefield derivative as defined in \Cref{sec:wave_equation_problem} on a regular grid. A high density of grid points ($\sim5\times$ spatial Nyquist frequency) is used to ensure that the simulation is high-fidelity.

\bibliographystyle{elsarticle-num}
\bibliography{refs}

\begin{thebibliography}{10}
\expandafter\ifx\csname url\endcsname\relax
  \def\url#1{\texttt{#1}}\fi
\expandafter\ifx\csname urlprefix\endcsname\relax\def\urlprefix{URL }\fi
\expandafter\ifx\csname href\endcsname\relax
  \def\href#1#2{#2} \def\path#1{#1}\fi

\bibitem{chen:2022:BTM}
J.~Chen, X.~Chi, W.~E, Z.~Yang, Bridging traditional and machine learning-based
  algorithms for solving pdes: the random feature method, J Mach Learn 1 (2022)
  268--98.

\bibitem{Anderson:2024:ELM}
S.~Anderson, V.~Dolean, B.~Moseley, J.~Pestana, {ELM-FBPINN}: efficient
  finite-basis physics-informed neural networks, arXiv preprint
  arXiv:2409.01949 (2024).

\bibitem{Lagaris1998}
I.~E. Lagaris, A.~Likas, D.~I. Fotiadis, {Artificial neural networks for
  solving ordinary and partial differential equations}, IEEE Transactions on
  Neural Networks 9~(5) (1998) 987--1000.
\newblock \href {http://arxiv.org/abs/9705023} {\path{arXiv:9705023}}, \href
  {https://doi.org/10.1109/72.712178} {\path{doi:10.1109/72.712178}}.

\bibitem{Raissi2019}
M.~Raissi, P.~Perdikaris, G.~E. Karniadakis,
  \href{https://doi.org/10.1016/j.jcp.2018.10.045}{{Physics-informed neural
  networks: A deep learning framework for solving forward and inverse problems
  involving nonlinear partial differential equations}}, Journal of
  Computational Physics 378 (2019) 686--707.
\newblock \href {https://doi.org/10.1016/j.jcp.2018.10.045}
  {\path{doi:10.1016/j.jcp.2018.10.045}}.
\newline\urlprefix\url{https://doi.org/10.1016/j.jcp.2018.10.045}

\bibitem{dissanayake:1994:NTB}
M.~G. Dissanayake, N.~Phan-Thien, Neural-network-based approximations for
  solving partial differential equations, Communications in Numerical Methods
  in Engineering 10~(3) (1994) 195--201.

\bibitem{Rahaman2018}
N.~Rahaman, A.~Baratin, D.~Arpit, F.~Draxlcr, M.~Lin, F.~A. Hamprecht,
  Y.~Bengio, A.~Courville, \href{http://arxiv.org/abs/1806.08734}{{On the
  spectral bias of neural networks}}, in: 36th International Conference on
  Machine Learning, ICML 2019, Vol. 2019-June, International Machine Learning
  Society (IMLS), 2019, pp. 9230--9239.
\newblock \href {http://arxiv.org/abs/1806.08734} {\path{arXiv:1806.08734}}.
\newline\urlprefix\url{http://arxiv.org/abs/1806.08734}

\bibitem{Moseley2023}
B.~Moseley, A.~Markham, T.~Nissen-Meyer,
  \href{https://link.springer.com/article/10.1007/s10444-023-10065-9}{{Finite
  basis physics-informed neural networks ({FBPINNs}): a scalable domain
  decomposition approach for solving differential equations}}, Advances in
  Computational Mathematics 2023 49:4 49~(4) (2023) 1--39.
\newblock \href {https://doi.org/10.1007/S10444-023-10065-9}
  {\path{doi:10.1007/S10444-023-10065-9}}.
\newline\urlprefix\url{https://link.springer.com/article/10.1007/s10444-023-10065-9}

\bibitem{Wang2021d}
S.~Wang, H.~Wang, P.~Perdikaris, {On the eigenvector bias of Fourier feature
  networks: From regression to solving multi-scale PDEs with physics-informed
  neural networks}, Computer Methods in Applied Mechanics and Engineering 384
  (2021) 113938.
\newblock \href {http://arxiv.org/abs/2012.10047} {\path{arXiv:2012.10047}},
  \href {https://doi.org/10.1016/j.cma.2021.113938}
  {\path{doi:10.1016/j.cma.2021.113938}}.

\bibitem{Moseley2022}
B.~Moseley, {Physics-informed machine learning: from concepts to real-world
  applications}, Ph.D. thesis, University of Oxford (2022).
\newblock \href {https://doi.org/10.13039/501100000266}
  {\path{doi:10.13039/501100000266}}.

\bibitem{jagtap2020conservative}
A.~D. Jagtap, E.~Kharazmi, G.~E. Karniadakis, Conservative physics-informed
  neural networks on discrete domains for conservation laws: Applications to
  forward and inverse problems, Computer Methods in Applied Mechanics and
  Engineering 365 (2020) 113028.

\bibitem{jagtap2020extended}
A.~D. Jagtap, G.~E. Karniadakis, Extended physics-informed neural networks
  ({XPINNs}): A generalized space-time domain decomposition based deep learning
  framework for nonlinear partial differential equations, Communications in
  Computational Physics 28~(5) (2020).

\bibitem{kharazmi2021hp}
E.~Kharazmi, Z.~Zhang, G.~E. Karniadakis, {hp-VPINNs}: Variational
  physics-informed neural networks with domain decomposition, Computer Methods
  in Applied Mechanics and Engineering 374 (2021) 113547.

\bibitem{li2020deep}
W.~Li, X.~Xiang, Y.~Xu, Deep domain decomposition method: Elliptic problems,
  in: Mathematical and Scientific Machine Learning, PMLR, 2020, pp. 269--286.

\bibitem{dolean2022finite}
V.~Dolean, A.~Heinlein, S.~Mishra, B.~Moseley, Finite basis physics-informed
  neural networks as a {Schwarz} domain decomposition method, in: International
  Conference on Domain Decomposition Methods, Springer, 2022, pp. 165--172.

\bibitem{Dolean:MDD:2024}
V.~Dolean, A.~Heinlein, S.~Mishra, B.~Moseley, Multilevel domain
  decomposition-based architectures for physics-informed neural networks,
  Computer Methods in Applied Mechanics and Engineering 429 (2024) 117116.

\bibitem{Huang2006}
G.~B. Huang, Q.~Y. Zhu, C.~K. Siew, {Extreme learning machine: Theory and
  applications}, Neurocomputing 70~(1-3) (2006) 489--501.
\newblock \href {https://doi.org/10.1016/j.neucom.2005.12.126}
  {\path{doi:10.1016/j.neucom.2005.12.126}}.

\bibitem{pao1992functional}
Y.-H. Pao, Y.~Takefuji, Functional-link net computing: theory, system
  architecture, and functionalities, Computer 25~(5) (1992) 76--79.

\bibitem{Jaeger:2009:RCA}
M.~Lukoševičius, H.~Jaeger,
  \href{https://www.sciencedirect.com/science/article/pii/S1574013709000173}{Reservoir
  computing approaches to recurrent neural network training}, Computer Science
  Review 3~(3) (2009) 127--149.
\newblock \href {https://doi.org/https://doi.org/10.1016/j.cosrev.2009.03.005}
  {\path{doi:https://doi.org/10.1016/j.cosrev.2009.03.005}}.
\newline\urlprefix\url{https://www.sciencedirect.com/science/article/pii/S1574013709000173}

\bibitem{datar:2024:PDE}
C.~Datar, T.~Kapoor, A.~Chandra, Q.~Sun, I.~Burak, E.~L. Bolager,
  A.~Veselovska, M.~Fornasier, F.~Dietrich,
  \href{https://arxiv.org/abs/2405.20836}{Solving partial differential
  equations with sampled neural networks} (2024).
\newblock \href {http://arxiv.org/abs/2405.20836} {\path{arXiv:2405.20836}}.
\newline\urlprefix\url{https://arxiv.org/abs/2405.20836}

\bibitem{Shang_Heinlein_Mishra_Wang_2025}
Y.~Shang, A.~Heinlein, S.~Mishra, F.~Wang, Overlapping schwarz preconditioners
  for randomized neural networks with domain decomposition, Computer Methods in
  Applied Mechanics and Engineering 442~(1) (Jul. 2025).
\newblock \href {https://doi.org/10.1016/j.cma.2025.118011}
  {\path{doi:10.1016/j.cma.2025.118011}}.

\bibitem{Chan_1987}
T.~F. Chan, Rank revealing qr factorizations, Linear Algebra and its
  Applications 88–89 (1987) 67–82.
\newblock \href {https://doi.org/10.1016/0024-3795(87)90103-0}
  {\path{doi:10.1016/0024-3795(87)90103-0}}.

\bibitem{Rahimi_Brecht_2007}
A.~Rahimi, B.~Brecht, Random features for large-scale kernel machines, Advances
  in Neural Information Processing Systems (2007).

\bibitem{Gu_Eisenstat_1996}
M.~Gu, S.~C. Eisenstat, Efficient algorithms for computing a strong
  rank-revealing qr factorization, SIAM Journal on Scientific Computing 17~(4)
  (1996) 848–869.
\newblock \href {https://doi.org/10.1137/0917055} {\path{doi:10.1137/0917055}}.

\bibitem{Paige1982}
C.~C. Paige, M.~A. Saunders,
  \href{https://dl.acm.org/doi/pdf/10.1145/355984.355989}{{LSQR: An Algorithm
  for Sparse Linear Equations and Sparse Least Squares}}, ACM Transactions on
  Mathematical Software (TOMS) 8~(1) (1982) 43--71.
\newblock \href
  {https://doi.org/10.1145/355984.355989/ASSET/CF3C8EE0-D698-4E67-BD8A-BF0B25681EEC/ASSETS/355984.355989.FP.PNG}
  {\path{doi:10.1145/355984.355989/ASSET/CF3C8EE0-D698-4E67-BD8A-BF0B25681EEC/ASSETS/355984.355989.FP.PNG}}.
\newline\urlprefix\url{https://dl.acm.org/doi/pdf/10.1145/355984.355989}

\bibitem{jax2018github}
J.~Bradbury, R.~Frostig, P.~Hawkins, M.~J. Johnson, C.~Leary, D.~Maclaurin,
  G.~Necula, A.~Paszke, J.~Vander{P}las, S.~Wanderman-{M}ilne, Q.~Zhang,
  \href{http://github.com/jax-ml/jax}{{JAX}: composable transformations of
  {P}ython+{N}um{P}y programs} (2018).
\newline\urlprefix\url{http://github.com/jax-ml/jax}

\bibitem{2020SciPy-NMeth}
P.~Virtanen, R.~Gommers, T.~E. Oliphant, M.~Haberland, T.~Reddy, D.~Cournapeau,
  E.~Burovski, P.~Peterson, W.~Weckesser, J.~Bright, S.~J. {van der Walt},
  M.~Brett, J.~Wilson, K.~J. Millman, N.~Mayorov, A.~R.~J. Nelson, E.~Jones,
  R.~Kern, E.~Larson, C.~J. Carey, {\.I}.~Polat, Y.~Feng, E.~W. Moore,
  J.~{VanderPlas}, D.~Laxalde, J.~Perktold, R.~Cimrman, I.~Henriksen, E.~A.
  Quintero, C.~R. Harris, A.~M. Archibald, A.~H. Ribeiro, F.~Pedregosa, P.~{van
  Mulbregt}, {SciPy 1.0 Contributors}, {{SciPy} 1.0: Fundamental Algorithms for
  Scientific Computing in Python}, Nature Methods 17 (2020) 261--272.
\newblock \href {https://doi.org/10.1038/s41592-019-0686-2}
  {\path{doi:10.1038/s41592-019-0686-2}}.

\bibitem{Stewart_1980}
G.~W. Stewart, The efficient generation of random orthogonal matrices with an
  application to condition estimators, SIAM Journal on Numerical Analysis
  17~(3) (1980) 403–409.
\newblock \href {https://doi.org/10.1137/0717034} {\path{doi:10.1137/0717034}}.

\bibitem{Tropp_2011}
J.~A. Tropp, A comparison principle for functions of a uniformly random
  subspace, Probability Theory and Related Fields 153~(3–4) (2011) 759–769.
\newblock \href {https://doi.org/10.1007/s00440-011-0360-9}
  {\path{doi:10.1007/s00440-011-0360-9}}.

\bibitem{DiaconisShahshahani1987}
P.~Diaconis, M.~Shahshahani, The subgroup algorithm for generating uniform
  random variables, Probability in the Engineering and Informational Sciences
  1~(1) (1987) 15--32.
\newblock \href {https://doi.org/10.1017/S0269964800000255}
  {\path{doi:10.1017/S0269964800000255}}.

\bibitem{Meckes2019}
M.~W. Meckes, \href{https://doi.org/10.1017/9781108642324}{The Random Matrix
  Theory of the Classical Compact Groups}, Vol. 218 of Cambridge Tracts in
  Mathematics, Cambridge University Press, 2019.
\newblock \href {https://doi.org/10.1017/9781108642324}
  {\path{doi:10.1017/9781108642324}}.
\newline\urlprefix\url{https://doi.org/10.1017/9781108642324}

\bibitem{Komatitsch2007}
D.~Komatitsch, R.~Martin, {An unsplit convolutional perfectly matched layer
  improved at grazing incidence for the seismic wave equation}, Geophysics
  72~(5) (2007) SM155--SM167.

\end{thebibliography}
\end{document}